\documentclass[reqno,11pt]{amsart}

\usepackage[english]{babel}
\usepackage[utf8]{inputenc}
\usepackage[T1]{fontenc}

\usepackage[margin=1in, letterpaper]{geometry}
\usepackage{amssymb,amsmath,amsfonts,amsthm}
\usepackage{bm,bbm,mathrsfs}
\usepackage[dvipsnames]{xcolor}
\usepackage{url}
\usepackage{enumitem}
\usepackage{mathtools}
\usepackage{placeins}
\usepackage{extarrows}
\usepackage{subfig}
\usepackage[font=small,labelfont=bf]{caption}
\usepackage{floatrow}
\usepackage{verbatim}
\usepackage{wrapfig}
\usepackage{tikz}
\usetikzlibrary{arrows,arrows.meta,cd,decorations.pathmorphing,backgrounds,positioning,fit,matrix}
\usepackage{xcolor,colortbl,graphics,graphicx}
\usepackage[colorlinks=true, allcolors=Blue, pdfencoding=auto]{hyperref}
\usepackage{aliascnt}
\usepackage[noabbrev,capitalize]{cleveref}

\newtheorem{theorem}{Theorem}[section]
\newaliascnt{lemma}{theorem}
\newtheorem{lemma}[lemma]{Lemma}
\aliascntresetthe{lemma}
\crefname{lemma}{Lemma}{Lemmas} 
\newaliascnt{proposition}{theorem}
\newtheorem{proposition}[proposition]{Proposition}
\aliascntresetthe{proposition}
\crefname{proposition}{Proposition}{Propositions}
\newaliascnt{corollary}{theorem}

\aliascntresetthe{corollary}
\crefname{corollary}{Corollary}{Corollaries}
\newaliascnt{example}{theorem}

\aliascntresetthe{example}
\crefname{example}{Example}{Examples}
\newaliascnt{question}{theorem}

\aliascntresetthe{question}
\crefname{question}{Question}{Questions}
\newaliascnt{conjecture}{theorem}

\aliascntresetthe{conjecture}
\crefname{conjecture}{Conjecture}{Conjectures}
\newaliascnt{assumption}{theorem}

\aliascntresetthe{assumption}
\crefname{assumption}{Assumption}{Assumptions}
\newaliascnt{definition}{theorem}

\aliascntresetthe{definition}
\crefname{definition}{Definition}{Definitions}
\newaliascnt{notation}{theorem}
\newtheorem{notation}[notation]{Notation}
\aliascntresetthe{notation}
\crefname{notation}{Notation}{Notations}
\theoremstyle{remark}
\newaliascnt{remark}{theorem}
\newtheorem{remark}[remark]{Remark}
\aliascntresetthe{remark}
\crefname{remark}{Remark}{Remarks}
\crefname{equation}{}{}
\numberwithin{equation}{section}
\allowdisplaybreaks

\newcommand{\Mult}{\mathrm{Mult}}

\newcommand{\bigboxtimesA}[2]{%
  \vcenter{\hbox{\scalebox{2}{$\displaystyle\boxtimes$}}}%
}
\newcommand{\lev}{\textnormal{lev}}

\newcommand{\Z}{\mathbb{Z}}
\newcommand{\R}{\mathbb{R}}

\newcommand{\C}{\mathbb{C}}

\newcommand{\F}{\mathbb{F}}

\renewcommand{\H}{\mathbb{H}}
\renewcommand{\P}{\mathbb{P}}

\newcommand{\msA}{\mathscr{A}}
\newcommand{\msB}{\mathscr{B}}
\newcommand{\msC}{\mathscr{C}}
\newcommand{\msS}{\mathscr{S}}
\newcommand{\msT}{\mathscr{T}}

\newcommand{\mI}{\mathcal{I}}
\newcommand{\mJ}{\mathcal{J}}

\newcommand{\ma}{\mathfrak{a}}

\newcommand{\one}{\mathbbm{1}}

\newcommand{\GL}{\textnormal{GL}}

\newcommand{\SL}{\textnormal{SL}}
\newcommand{\PSL}{\textnormal{PSL}}

\newcommand{\Tr}{\textnormal{Tr}}

\newcommand{\lf}{\left\lfloor}
\newcommand{\rf}{\right\rfloor}

\newcommand{\eps}{\varepsilon}

\renewcommand{\bar}{\overline}
\renewcommand{\hat}{\widehat}

\renewcommand{\pmod}[1]{\ (\textnormal{mod } #1)}

\definecolor{myBlue}{rgb}{0, 0, 0.5}
\definecolor{myRed}{rgb}{0.6, 0, 0}
\definecolor{myGreen}{rgb}{0, 0.4, 0}

\newcommand{\ap}[1]{\color{myGreen} {\bf Alex:} #1 \color{black}}

\title{Bilinear forms with Kloosterman sums via quadratic characters}

\author[Valentin Blomer]{Valentin Blomer}
\address{Mathematisches Institut, Endenicher Allee 60, 53115 Bonn, Germany}
\email{blomer@math.uni-bonn.de}

\author[Alexandru Pascadi]{Alexandru Pascadi}
\address{Mathematisches Institut, Endenicher Allee 60, 53115 Bonn, Germany}
\email{pascadi@math.uni-bonn.de}

\thanks{
Both authors supported through EXC-2047/1 - 390685813
and by ERC Advanced Grant 101054336. }



\begin{document}

\begin{abstract}
   We prove new bounds for bilinear forms with Kloosterman sums, valid for all moduli $c$. In the critical range where the summation length is the square root of the modulus, the saving over the trivial bound is $c^{-1/32}$, improving on all previous approaches even for prime moduli. This is based on a new connection to quadratic character sums. Applications to moments of twisted $L$-functions and to the large sieve for exceptional Maa{\ss} forms are given.
\end{abstract}

\maketitle


\section{Introduction} \label{sec:mom-sv}

\subsection{The main result}
Kloosterman sums are ubiquitous in number theory. Weil's bound provides a best possible estimate, but often one is interested in sums of Kloosterman sums against various other sequences, over the arguments, the modulus or both. A recurring theme are bilinear forms with Kloosterman sums of the shape 
$$\sum_{m=1}^N \sum_{n=1}^N \alpha_m \beta_n S(m, n; c) $$
where $\alpha$ and $\beta$ are arbitrary sequences. Using Weil's bound for Kloosterman sums, the trivial bound is
\begin{equation}\label{eq:trivial-Weil} 
    \| \alpha \| \| \beta \| N c^{1/2 + o(1)},
\end{equation}
where here and henceforth $\| \cdot \|$ denotes the $2$-norm, i.e.\ $\| \alpha \| = (\sum_n |\alpha_n|^2)^{1/2}$. 
On the other hand,  opening the Kloosterman sum and applying Cauchy's inequality (see \cite[p.\ 82]{iwaniec1997topics}) gives
\begin{equation}\label{eq:trivial-Fourier} 
\| \alpha \| \| \beta\| (c+N).
\end{equation}
A critical range is $N \approx \sqrt{c}$, in which case the two previous bounds coincide. This is often the threshold in applications. A lot of work has been devoted to improving the trivial bound by diverse methods coming from algebraic geometry, $p$-adic analysis and additive combinatorics (see in particular \cite{kowalski2017bilinear, kowalski2020stratification, blomer2015second, milicevic2025bilinear, pascadi2025nonabelian, fouvry2025bilinear}). In this paper we present a new method that is superior to all of these bounds in this range, and our main result applies to arbitrary moduli.

\begin{theorem} \label{thm:bilinear-forms-balanced-lengths}
Let $c \in \Z_+$, $N \in \Z \cap [1, c]$, and $\mI, \mJ \subset \Z$ be intervals with $|\mI|, |\mJ| \le N$. Then for any complex sequences $(\alpha_m)_{m \in \mI}$, $(\beta_n)_{n \in \mJ}$ and any $a \in (\Z/c\Z)^\times$, one has
\begin{equation}\label{eq:bilinear-forms-balanced-lengths}
\begin{aligned}
    \mathop{\sum\sum}_{\substack{m \in \mI, n \in \mJ \\ (m, n, c) = 1}} \alpha_m \beta_n S(am, n; c) 
    &\ll
    \|\alpha\| \|\beta\| c^{1+o(1)}
    \Big(\frac{N^{1/8}}{c^{3/32}} + \frac{N^{5/16}}{c^{3/16}} + \frac{N^{2/3}}{c^{7/18}}\Big)
    \\
    &=
    \|\alpha\| \|\beta\|
    Nc^{1/2 + o(1)}
    \Big(\frac{c^{13/32}}{N^{7/8}} + \frac{c^{5/16}}{N^{11/16}} + \frac{c^{1/9}}{N^{1/3}}\Big).
\end{aligned}
\end{equation}
If $\mI = \mJ = \{1, \ldots, N\}$, then \cref{eq:bilinear-forms-balanced-lengths} also holds without the constraint $(m, n, c) = 1$.
\end{theorem}

\begin{remark} A more complicated formula for $\mI$ and $\mJ$ of different lengths is given in \cref{thm:bilinear-forms-general} below, which in some cases can be complemented with  \cref{polyaV}, \cref{lem:bilinear-forms-trivial-bound} and  \cref{largesize}. 

In the critical range  $N = \sqrt{c}$, the bound in \cref{thm:bilinear-forms-balanced-lengths} reads
\[
\begin{aligned}
    \mathop{\sum\sum}_{\substack{m \in \mI, n \in \mJ \\ (m, n, c) = 1}} \alpha_m \beta_n S(am, n; c) 
    &\ll
    \|\alpha\| \|\beta\| c^{1-\frac{1}{32}+o(1)}.
\end{aligned}
\]
This is new even for prime moduli $c = p$; it doubles the saving $p^{-1/64}$ of Kowalski--Michel--Sawin \cite{kowalski2017bilinear} (an improvement of the method in \cite{kowalski2020stratification} may give $p^{-1/60}$). For general moduli $c$, the recent  work of Mili\'cevi\'c--Qin--Wu \cite{milicevic2025bilinear} saves $c^{-1/100}$ in the square-root range. 

Our result beats the trivial bound $\|\alpha\|\|\beta\| c^{o(1)} \min(c, N\sqrt{c})$ in the range
\[
    N \in (c^{13/28+\eps}, c^{7/12-\eps}).
\]
The lower range $N > c^{13/28} = c^{1/2 - 1/28}$ 
improves both on \cite{milicevic2025bilinear}, which gives $N > c^{1/2 - 1/42}$, 
and on \cite[Theorem 7.8]{pascadi2025nonabelian}, which gives $N > c^{1/2 - 1/82}$; for prime $c = p$, the record remains $N > p^{3/8}$ due to Kowalski--Michel--Sawin \cite{kowalski2020stratification}. The upper range $N < c^{7/12}$ is to our knowledge also the best one available for general moduli.  
Thus, for general $c$,  \eqref{eq:bilinear-forms-balanced-lengths} gives the best result in all important aspects.

Note that in general the condition $(m, n, c) = 1$ in \eqref{eq:bilinear-forms-balanced-lengths} cannot be dropped, for instance if $c$ is even and $\mI =  \mJ = \{c/2\}$. This is different from the situation where bilinear forms in $S(mn, 1, c)$ are considered. In applications, it is often easy to switch from one setting to the other. 

The proof of \cref{thm:bilinear-forms-balanced-lengths} relies primarily on a new connection to sums of quadratic Dirichlet characters; the argument for prime moduli is sketched in \cref{subsec:sketch-prime-moduli}. To obtain a uniform result for general moduli, we combine these ideas with the approach of \cite{pascadi2025nonabelian}, which uses the representation theory of $\SL_2(\Z/c\Z)$; this is possible due to an interplay between certain characters of $\SL_2(\Z/c\Z)$ and $(\Z/c\Z)^\times$, as explained in \cref{subsec:sketch-hybrid-approach}.  In particular, \cref{thm:bilinear-forms-balanced-lengths} does not rely on the results of Kowalski--Michel--Sawin \cite{kowalski2017bilinear,kowalski2020stratification}, Blomer--Mili\'cevi\'c \cite{blomer2015second}, or Mili\'cevi\'c--Qin--Wu \cite{milicevic2025bilinear}.

\subsection{Application 1: moments of twisted \texorpdfstring{$L$}{L}-functions} One of the key applications of bilinear forms with Kloosterman sums is an asymptotic formula for the second moment of twisted $L$-functions. In fact, this problem prompted   the breakthrough result \cite{kowalski2017bilinear} for prime moduli. Here 
prove the following result. If $\lambda_1(n)$, $\lambda_2(n)$ denote the Hecke eigenvalues of two cusp forms $f_1$, $f_2$, we define the following Euler products:
\begin{align*}
  & P(s)
  = \prod_{p \mid q} 
  \Big(1 - \frac{\lambda_1(p^2)}{p^s} + \frac{\lambda_1(p^2)}{p^{2s}} - \frac{1}{p^{3s}}\Big) \Big(1 - \frac{1}{p^{2s}}\Big)^{-1},\\ \nonumber
  & Q(s) 
  = \prod_{p \mid q} 
  \Big(1 - \frac{\lambda_1(p )\lambda_2(p )}{p^s} + \frac{\lambda_1(p^2 ) + \lambda_2(p^2 ) }{p^{2s}} - \frac{\lambda_1(p )\lambda_2(p )}{p^{3s}}  + \frac{1}{p^{4s}}\Big) \Big(1 - \frac{1}{p^{2s}}\Big)^{-1}.   \nonumber
\end{align*}
For a positive integer $q \not\equiv 2$ (mod 4) let $\phi^{\ast}(q)$ denote the number of primitive characters modulo $q$.

\begin{theorem}\label{secondmoment} For $j = 1, 2$ let $f_j$ be both holomorphic or both Maa{\ss} cuspidal Hecke eigenforms for ${\rm SL}_2(\mathbb{Z})$. 
Assume that their root numbers satisfy $\epsilon(f_1)\epsilon(f_2) = 1$. Then for any $q \not\equiv 2 \pmod{4}$ and any $\varepsilon > 0$ we have
$$\frac{1}{\phi^*(q)} \sum_{\substack{\chi \pmod{q}\\ \chi \text{ {\rm  primitive}}}} L(1/2, f_1 \times \chi) L(1/2, f_2 \times \bar{\chi}) = \frac{2}{\zeta(2)} M(f_1, f_2, q) + O(q^{-1/90 + \varepsilon})$$
where
$$M(f_1, f_2, q) =  \begin{cases}
\displaystyle P(1) L(1, {\rm sym}^2 f_1) \Big( \log q + c +\frac{P'(1)}{P(1)}   \Big), &f_1=f_2,\\
Q(1) L(1, f_1 \times f_2), &f_1\neq f_2, \end{cases}$$
and $c$ is a constant depending only on $f_1$ (not on $q$). 
\end{theorem}

\begin{remark}
    The error term is better than the best known result $q^{-1/144+\varepsilon}$ even in the prime case \cite{blomer2017moments, kowalski2017bilinear}, and improves for general moduli the recent 
    $q^{-1/216+\varepsilon}$    \cite[Theorem 1.2]{milicevic2025bilinear}. In special cases, e.g.\ if $q$ is squarefree and/or if $f_1, f_2$ are holomorphic (and hence the Ramanujan-Petersson conjecture is known), the exponent can be further improved. 

    The mixed case where $f_1$ is holomorphic and $f_2$ is Maa{\ss} can be handled by a small modification of the analysis in \cite[Section 7 \& 8]{blomer2015second} as in \cite[Section 3]{blomer2017moments}. 
 \end{remark}

\end{remark}

\subsection{Application 2: large sieve for exceptional Maa{\ss} forms}

Our second application concerns the $\GL_2$ spectral large sieve, as pioneered by Deshouillers--Iwaniec \cite{deshouillers1982kloosterman}. Via the Kuznetsov formula \cite{kuznetsov1980petersson}, the spectral large sieve leads to bounds for multilinear forms with Kloosterman sums, which have been central to many developments in analytic number theory -- concerning for instance moments of $L$-functions and shifted convolution problems \cite{deshouillers1982power,deshouillers1984power,blomer2017moments,topacogullari2018shifted,chandee2024eighth}, prime factors of polynomials \cite{deshouillers1982greatest,merikoski2023largest,de2020niveau,pascadi2026large,grimmelt2025greatest}, and the distribution of primes in arithmetic progressions \cite{bombieri1986primes,bombieri1987primes2,bombieri1989primes3,maynard2025primes,maynard2025primes2,maynard2025primes3,lichtman2025modification,pascadi2025exponents}. In many of these applications, special care is required in the exceptional spectrum, containing those Maa{\ss} cusp forms which might fail Selberg's eigenvalue conjecture \cite{selberg1965estimation}, i.e., which have eigenvalues $\lambda < \tfrac{1}{4}$ with respect to the hyperbolic Laplacian. We find it convenient to introduce the following notation; we point the reader to \cref{sec:large-sieve} for more details and background.

\begin{notation} \label{not:orthonormal-basis}
Let $q = rs$ where $r, s \in \Z_+$ and $(r, s) = 1$. Let $\ma$ be a cusp of $\Gamma_0(q)\backslash \H$ equivalent to $\tfrac{1}{s}$, and $\sigma_\ma \in \PSL_2(\R)$ be a scaling matrix for $\ma$. Let $(f_j)_{j \ge 1}$ be a complete orthonormal basis of Maa{\ss} cusp forms $f_j : \Gamma_0(q)\backslash\H \to \C$, with Fourier coefficients $(\rho_{j\ma}(n))_{n \in \Z}$ normalized as in \cref{eq:Maass-Fourier-expansion}, and Laplacian eigenvalues $\lambda_j$. We write $\theta_j := \max(0, \tfrac{1}{4}-\lambda_j)^{1/2}$.
\end{notation}

Selberg's eigenvalue conjecture asserts that $\theta_j = 0$, but the best unconditional bound is $\theta_j \le \tfrac{7}{64}$, due to Kim--Sarnak \cite[Appendix 2]{kim2003functoriality}.
The role of exceptional-spectrum large sieve inequalities is to temper the contribution of a factor $X^{2\theta_j}$ on average over a basis, where $X$ is as large as possible; as in \cite[\S 9]{pascadi2025nonabelian}, we can use improved bounds for bilinear forms with Kloosterman sums to achieve larger values of $X$.
Compared to \cite[Theorem 9.3]{pascadi2025nonabelian}, \cref{thm:large-sieve} below gives a uniform result that does not assume a special factorization of the level $q$, and removes the coprimality constraint $(n, q) = 1$.

\begin{theorem}\label{thm:large-sieve}
Let $N \ge \tfrac{1}{2}$ and $(\alpha_n)_{n \sim N}$ be a complex sequence. Using \cref{not:orthonormal-basis}, one has
\[
    \sum_{\lambda_j < 1/4}
    X^{2\theta_j} 
    \Big\vert 
    \sum_{n \sim N} \alpha_n\, \rho_{j\ma}(n)
    \Big\vert^2 
    \ll
    (qN)^{o(1)}
    \Big(1 + \frac{N}{q}\Big) \|\alpha\|^2,
\]
with
\begin{equation}\label{eq:X-choice}
    X = 1 + \frac{q}{N} + 
    \min\Big(\frac{q^{18/11}}{N^{23/11}}, 
    \frac{q^{16/13}}{N^{18/13}},
    \frac{q^{32/29}}{N^{33/29}}
    \Big)
    + 
    \frac{q^2}{N^3}.
\end{equation}
\end{theorem}

\begin{remark}
In the results of Deshouillers--Iwaniec \cite[Theorems 2 and 5]{deshouillers1982kloosterman}, one can take $X = 1 + \tfrac{q}{N}$. The value of $X$ in \cref{eq:X-choice} can be improved if the level $q$ has a suitable factorization \cite[Corollary 1.6]{pascadi2025nonabelian}, or if the sequence $(\alpha_n)$ has a special structure \cite{pascadi2026large}, but the main significance of \cref{thm:large-sieve} is obtaining a non-trivial improvement over Deshouillers--Iwaniec in the general setting.

The order of the terms in \cref{eq:X-choice} matches their relevance as $N$ decreases (e.g., the $1$ term dominates when $N > q$, and the $ q^2/N^3$ term dominates when $N < q^{13/27}$). When $N = \sqrt{q}$, the relevant term is $ q^{32/29}/N^{33/29}$ and we have $X \asymp q^{1/2 + 1/29}$, so \cref{thm:large-sieve} saves $q^{2\max_j \theta_j/29}$ over \cite[Theorem 5]{deshouillers1982kloosterman}. 
\end{remark}

\subsection{Sketch of the argument for prime moduli} \label{subsec:sketch-prime-moduli}
In this subsection and the next, we ignore various technical details including smooth weights and $c^{o(1)}$ factors, and we use `$\approx$', `$\lesssim$' to indicate the approximate nature of the estimates. 

Let us informally sketch the proof of \cref{thm:bilinear-forms-balanced-lengths} when $c = p$ is a prime, $N \approx \sqrt{p}$, $\mI = \mJ = (N, 2N] \cap \Z$, and $a = 1$. By two applications of Cauchy--Schwarz (in $m$ and then in the two copies of the $n$ variable), we are left to bound 
\begin{equation}\label{eq:sketch-4th-moment}
    \sum_{n_1, n_2, n_3, n_4 \sim \sqrt{p}}
    S(n_1, n_2; p) S(n_2, n_3; p) S(n_3, n_4; p) S(n_4, n_1; p)
    \lesssim
    p^{4-4\delta},
\end{equation}
with $\delta = \tfrac{1}{32}$.
Note that the pointwise Weil bound still matches the trivial bound where $\delta = 0$.

We expand the Kloosterman sums and swap sums to reach
\[
    \sum_{x_1, x_2, x_3, x_4 \in \F_p^\times} 
    \prod_{j \pmod{4}} \sum_{n_j \sim \sqrt{p}} e\Big(\frac{n_j (x_j + \bar x_{j-1})}{p}\Big)
    \lesssim p^{4-4\delta}.
\]
Evaluating the linear sums over $n_j$, writing $h_j = x_j + \bar x_{j-1}$, and taking out the zero frequencies $h_j = 0$ leaves us with essentially
\[
    \sum_{h_1, h_2, h_3, h_4 \sim \sqrt{p}} 
    \Big( \#\{x_1, x_2, x_3, x_4 \in \F_p^\times : \forall j \pmod{4}: x_j + \bar x_{j-1} = h_j\}
    - 1
    \Big)
    \lesssim p^{2-4\delta},
\]
where we could subtract a main term of $1$ due to cancellation in the underlying smooth weights of the $h_j$-variables (coming from the fact that the dual variables $n_j$ are nonzero). Substituting $x_2, x_3, x_4$ in the equations $x_j + \bar x_{j-1} = h_j$ leads to a quadratic congruence in $x_1 \in \F_p^\times$,
to which the number of solutions is typically one plus the Legendre symbol of the discriminant. This ultimately brings us to a quadratic character sum of the shape
\begin{equation}\label{eq:sketch-char-sum}
    \sum_{h_1, h_2, h_3, h_4 \sim \sqrt{p}}
    \Big(\frac{(h_1h_2h_3h_4 - (h_1+h_3)(h_2+h_4) + 2)^2 - 4}{p}\Big)
    \lesssim p^{2-4\delta}.
\end{equation}
Once again, the trivial bound corresponds to $\delta = 0$, so we only need a small amount of cancellation.

We then group $h_1, h_2, h_3$ to create complete variables $x := h_1h_2h_3 - h_1 - h_3$ and $y := 2-(h_1+h_3)h_2$. To amend the sparse support of $(x; y) \in \F_p^2$ (roughly of size $p^{3/2}$), we apply H\"older's inequality with parameters $\tfrac{1}{2}$, $\tfrac{1}{4}$, $\tfrac{1}{4}$. This bounds the left-hand side of \cref{eq:sketch-char-sum} by $\msA^{1/2} \msB^{1/4} \msC^{1/4}$, where 
\begin{gather*}
    \msA := \sum_{x, y \pmod{p}} \sum_{\substack{h_1, h_2, h_3 \sim \sqrt{p} \\ h_1h_2h_3-h_1-h_3 \equiv x \pmod{p} \\ 2-(h_1+h_3)h_2 \equiv y \pmod{p}}} 1,
    \qquad\quad 
    \msB := \sum_{x, y \pmod{p}} \Big\vert\sum_{\substack{h_1, h_2, h_3 \sim \sqrt{p} \\ h_1h_2h_3-h_1-h_3 \equiv x \pmod{p} \\ 2-(h_1+h_3)h_2 \equiv y \pmod{p}}} 1\Big\vert^2,
    \\
    \msC := 
    \sum_{x, y \pmod{p}} 
    \Big\vert\sum_{h_4 \sim \sqrt{p}} \Big(\frac{(xh_4+y)^2-4}{p}\Big) \Big\vert^4.
\end{gather*}
We immediately have $\msA \approx p^{3/2}$, and an elementary treatment gives the sharp bound $\msB \lesssim p^{3/2}$ (a more refined analysis, involving the Weil bound for Kloosterman sums, leads to a better bound for $\msB$ when the initial length is $N < \sqrt{p}$; see \cref{prop:B-sum-bound-alt}). After expanding the square in $\msC$, the diagonal terms contribute $\approx p^3$, while the off-diagonal terms exhibit square-root cancellation in $x, y$ due to Weil's bound for quadratic character sums, giving $\msC \lesssim p^3$. Combining these estimates produces an admissible bound in \cref{eq:sketch-char-sum} as long as
\[
    p^{\frac{3}{2} \cdot \frac{1}{2}} \cdot p^{\frac{3}{2} \cdot \frac{1}{4}} \cdot p^{3 \cdot \frac{1}{4}} \le p^{2-4\delta}
    \qquad \iff \qquad 
    \delta \le \frac{1}{32}.
\]

\subsection{A hybrid approach for general moduli} \label{subsec:sketch-hybrid-approach}

Generalizing the argument in \cref{subsec:sketch-prime-moduli} requires a more abstract formulation.
The sum in \cref{eq:sketch-4th-moment} is the fourth moment of eigenvalues of the matrix $(S(m, n; p))_{m, n \sim \sqrt{p}}$, which gives a natural upper bound for (the fourth power of) its spectral norm. In \cite{pascadi2025nonabelian}, the more general matrix $K_c := (S(m, n; c) \one_{(m, n, c) = 1})_{m, n \sim \sqrt{c}}$ was related to the values of a certain representation $\rho_c^\circ$ of $\SL_2(\Z/c\Z)$. Writing $\chi_c^\circ := \Tr\, \rho_c^\circ$, we can very roughly expand
\[
    \Tr(K_c^4) \approx 
    c^2
    \sum_{|h_1|,|h_2|,|h_3|,|h_4| \le \sqrt{c}} \chi_c^\circ(T^{h_1}S T^{h_2}S T^{h_3}S T^{h_4}S),
    \qquad 
    T := \begin{pmatrix} 1 & 1 \\ 0 & 1 \end{pmatrix},\ 
    S := \begin{pmatrix} 0 & -1 \\ 1 & 0 \end{pmatrix}.
\]
If $c = p$ is a prime, then for any $g \in \SL_2(\F_p)$, $\chi_p^\circ(g)$ can be written explicitly in terms of the Legendre symbol modulo $p$ of the discriminant of $g$ (see \cref{eq:conn-to-quadratic-prime}); this brings us to the same character sum as in \cref{eq:sketch-char-sum}.
A similar argument applies if $c$ is square-free.

However, finding an explicit expression for $\chi_c^\circ(g)$ in terms of quadratic characters becomes difficult in the depth aspect, i.e., when $c$ is divisible by large powers of primes. We circumvent this issue using a non-abelian amplification step, borrowing ideas from \cite{pascadi2025nonabelian}. This leaves us with a character sum modulo the square-free part of $c$, at the cost of larger diagonal terms (see \cref{eq:amplification}).

The result of the argument outlined so far is given in \cref{thm:bilinear-forms-4th-moment}, which works well when the square-full part of $c$ is not too large. To handle the remaining ranges, we directly use the bound for bilinear forms from \cite[Theorem 7.1]{pascadi2025nonabelian}, which depends on the factorization of the modulus (notably, the starting point of this bound is a sixth rather than a fourth moment of eigenvalues).

\begin{remark} 
The bound \cite[Theorem 7.1]{pascadi2025nonabelian} cannot obtain a saving over the trivial bounds when $c = p$ is prime. The connection between $\chi_p^\circ$ and the Legendre symbol is a key new input in our work, and it is the refined analysis of the resulting character sums that makes such savings possible.
\end{remark}

\begin{remark}
Using higher moments of eigenvalues in \cref{eq:sketch-4th-moment}, or different choices of parameters in the application of H\"older's inequality, does not appear to improve our final results in the critical ranges.
\end{remark}

\section{Notation and preliminaries} \label{sec:prelims}


\subsection{General notation}
We use the standard asymptotic notation from analytic number theory, indicating all dependencies of implicit constants through subscripts (when no parameter is indicated, the implicit constant is absolute). 
In expressions like $x^{o(1)}$, the implicit function depends on $x$ and is understood to be real; thus $f(x) \ll x^{o(1)}$ is equivalent to $\forall \eps > 0 : f(x) \ll_\eps x^\eps$, and $f(x) \gg x^{-o(1)}$ is equivalent to $\forall \eps > 0 : f(x) \gg_\eps x^{-\eps}$. The signs of the $o(1)$ quantities do not influence the formal meaning of such statements, but they are suggestive of the translation to the $\eps$-statements. 

We write $\Z_+$ for the set of positive integers. An interval $\mI \subset \Z$ refers to a set of consecutive integers. We use the notation $n \sim N$ for $n \in \Z \cap (N, 2N]$. Given $c \in \Z_+$ and $m, n \in \Z$ or $\Z/c\Z$, we write $(n, c)$ and $(m, n, c)$ for their greatest common divisors; note that these only depend on the residues of $m, n \pmod{c}$. We mention two more unusual pieces of notation associated with the greatest common divisor. In Section \ref{sec6} we write $(a, b^{\infty})$ to mean $\lim_{n \rightarrow \infty} (a, b^n)$. Lemma \ref{lem-expsum} features the following notation: For an integer $g \not= 0$ and a positive integer $c = \prod_p p^{e_p}$, we write
\begin{equation}\label{sqrtgcd}
(g, c^{1/2}) := \prod_p p^{\min(v_p(g),   e_p/2)}.
\end{equation}

For odd $c \in \Z_+$ and $a \in \Z$, we define the Jacobi symbol by
\begin{equation}\label{eq:jacobi-symbol}
    \Big(\frac{a}{c}\Big) := \prod_{\substack{\text{prime }p 
    \\ p^k \mid c,\ p^{k+1} \nmid c}} \Big(\frac{a}{p}\Big)^k,
\end{equation}
where $(\tfrac{a}{p})$ is the Legendre symbol.
For $n \in \Z_+$, we write $\tau(n)$ for the divisor-counting function obeying the divisor bound $\tau(n) \ll n^{o(1)}$, and we let
\begin{equation}\label{eq:nu-function}
    \nu_n := \prod_{\text{prime }p \mid n} \frac{-1}{p^2-1}.
\end{equation}
This is a weight that arises naturally in \cite[Proposition 4.10]{pascadi2025nonabelian} and therefore makes its way into some of our estimates.

Given $c \in \Z_+$ and $x \in (\Z/c\Z)^\times$ (or $x \in \Z$ with $(x, c) = 1$), we write $\bar{x} \in \Z$ for an inverse of $x$ modulo $c$.
For $t \in \R$, we also write $e(t) := \exp(2\pi i t)$. So for $m, n \in \Z$ and $c \in \Z_+$, one can write the Kloosterman and Ramanujan sums as
\[
    S(m, n; c) := \sum_{x \in (\Z/c\Z)^\times} e\Big(\frac{mx + n\bar{x}}{c}\Big),
    \qquad\qquad 
    r_m(c) := S(m, 0; c).
\]
The Kloosterman sums satisfy the Weil bound (see e.g. \cite[Corollary 11.12]{iwaniec2004analytic})
\begin{equation}\label{eq:Weil-bound}
    |S(m, n; c)| \le \tau(c) \sqrt{(m, n, c) c}.
\end{equation}

The spectral norm of a matrix $A \in \C^{m \times n}$ is
\begin{equation}\label{eq:spectral-norm}
    \|A\| = \max_{\substack{v \in \C^n \\ \|v\| \le 1}} \|Av\|
    =
    \max_{\substack{v \in \C^n,\, w \in \C^m \\ \|v\|, \|w\| \le 1}} |w^TAv|.
\end{equation}
Writing $A^*$ for the conjugate transpose of $A$, we have
\begin{equation}\label{eq:trace-method}
    \|A\| \le \Tr\Big((AA^*)^k\Big)^{\frac{1}{2k}}
\end{equation}
for any $k \in \Z_+$.

We write the Fourier transform of an $L^1$-function $f : \R \to \C$ as 
\[
    \hat{f} : \R \to \C,
    \qquad 
    \hat{f}(\xi) := 
    \int_{-\infty}^\infty f(t)\, e(-t\xi) dt.
\]
If $f$ is a Schwartz function, then so is $\hat{f}$, and Poisson summation reads $\sum_{n \in \Z} f(n) = \sum_{n \in \Z} \hat{f}(n)$.

We write $U(V)$ for the space of unitary transformations of a finite-dimensional complex Hilbert space $V$.
Given a finite group $G$ and a (not necessarily irreducible) representation $\rho : G \to U(V)$, we write $\dim \rho := \dim V$. Given a function $f : G \to \C$, we write $\hat{f}(\rho) : V \to V$ for the map
\begin{equation}\label{eq:Fourier-tr-nonab}
    \hat{f}(\rho) := \sum_{g \in G} f(g) \rho(g).
\end{equation}
We also write $\hat{G}$ for a complete set (up to isomorphism) of irreducible unitary representations of $G$. Given $\rho' \in \hat{G}$, we write $\Mult(\rho', \rho)$ for the multiplicity of $\rho'$ in $\rho$. Note that as characters of $G$ we have 
\begin{equation}\label{eq:trace-dec}
    \Tr \rho = \sum_{\rho' \in \hat{G}} \Mult(\rho', \rho) \Tr \rho'. 
\end{equation}
If $\rho = R$ is the (left-)regular representation $G$, we have $\Mult(\rho', R) = \dim \rho'$.

\subsection{\texorpdfstring{$\SL_2$}{SL2} and its representation theory} 
We write
\begin{equation}\label{eq:generators}
    I := \begin{pmatrix} 1 & 0 \\ 0 & 1 \end{pmatrix},
    \qquad 
    T := \begin{pmatrix} 1 & 1 \\ 0 & 1 \end{pmatrix},
    \qquad 
    S := \begin{pmatrix} 0 & -1 \\ 1 & 0 \end{pmatrix}
    \qquad 
    \in \SL_2(\Z),
\end{equation}
and note that $S^2 = -I$. Any congruence of $2 \times 2$ matrices refers to a set of $4$ entry-wise congruences, so for example $T^c \equiv I \pmod{c}$.

For $c \in \Z_+$, we recall that $|\SL_2(\Z/c\Z)| \asymp c^3$, and we let
\[
    \PSL_2(\Z/c\Z) := \SL_2(\Z/c\Z)/Z(\SL_2(\Z/c\Z)),
\]
where
\begin{equation}\label{eq:center}
    Z(\SL_2(\Z/c\Z)) = \left\{\gamma I : \gamma \in \Z/c\Z, \gamma^2 = 1\right\},
    \qquad 
    |Z(\SL_2(\Z/c\Z))| \ll c^{o(1)}.
\end{equation}
We also recall the projective line modulo $c$: if $\sim$ denotes the equivalence relation generated by $(a; b) \sim (\alpha a, \alpha b)$, for $\alpha \in (\Z/c\Z)^\times$, we have
\[
    \P^1(\Z/c\Z) := \left\{(a; b) \in (\Z/c\Z)^2: (a, b, c) = 1 \right\}/_\sim.
\]
The group $\PSL_2(\Z/c\Z)$ (and, through it, $\SL_2(\Z/c\Z)$), acts on $\P^1(\Z/c\Z)$ by M\"obius transformations.

\begin{notation}[Special representations \cite{pascadi2025nonabelian}] \label{not:non-abelian-chars}
For $c \in \Z_+$, we borrow from \cite[\S 4.1]{pascadi2025nonabelian} the construction of certain representations of $\SL_2(\Z/c\Z)$. Specifically, $\rho_c$ is the permutation representation associated to the action of $\SL_2(\Z/c\Z)$ on $\P^1(\Z/c\Z)$ (see \cite[Definition 4.1]{pascadi2025nonabelian}), and $\rho_c^\circ$ is a certain subrepresentation of $\rho_c$ (see \cite[Definition 4.5]{pascadi2025nonabelian}). We write $\chi_c := \Tr\, \rho_c$ and $\chi_c^\circ := \Tr\, \rho_c^\circ$.
We may evaluate $\rho_c, \rho_c^\circ, \chi_c, \chi_c^\circ$ at elements of $\SL_2(\Z)$ (or $\SL_2(\Z/c'\Z)$ where $c \mid c'$) by reducing them modulo $c$.
\end{notation}

\begin{remark}
The construction of $\rho_c^\circ$ in \cite[\S 4.1]{pascadi2025nonabelian} is non-trivial (it is the restriction of $\rho_c$ to a tensor product of local pieces, each of which is the complement of a certain fixed-point space). The key idea is that $\rho_c^\circ$ `sifts out' the contribution of small-dimensional irreducible representations to $\rho_c$; in particular, the trivial representation has multiplicity one in $\rho_c$ and zero in $\rho_c^\circ$ (unless $c = 1$). We will mainly need the following properties, which appear or are implicit in \cite{pascadi2025nonabelian}.
\end{remark}

\begin{lemma}[Properties of special representations \cite{pascadi2025nonabelian}] \label{lem:non-abelian-chars}
Let $g \in \SL_2(\Z)$ and $c, c_1, c_2 \in \Z_+$.
\begin{itemize}
    \item[$(i)$.] If $c = c_1c_2$ with $(c_1, c_2) = 1$, then $\chi_c(g) = \chi_{c_1}(g) \chi_{c_2}(g)$ and $\chi_c^\circ(g) = \chi_{c_1}^\circ(g) \chi_{c_2}^\circ(g)$.
    \item[$(ii)$.] $\sum_{d \mid c} \chi_d^\circ(g) = \chi_c(g)$ is the number of fixed points of $g$ in $\P^1(\Z/c\Z)$.
    \item[$(iii)$.] $\dim \rho_c^\circ \asymp c$, and all irreducible subrepresentations of $\rho_c^\circ$ have dimensions $\gg c^{1-o(1)}$.
\end{itemize}
\end{lemma}

\begin{proof}
Part $(i)$ follows directly by taking traces in \cite[(4.3) and (4.5)]{pascadi2025nonabelian}. The identity $\sum_{d \mid c} \chi_d^\circ = \chi_c(g)$ from part $(ii)$ reduces to a local statement by part $(i)$. Locally, the identity
\[
    \sum_{j=0}^k \chi_{p^j}^\circ(g) = \chi_{p^k}(g),
    \qquad 
    \forall p \text{ prime}, k \in \Z, k \ge 0
\] 
follows from the facts that $\chi_1(g) = \chi_1^\circ(g) = 1$ (since we identify an empty tensor product of Hilbert spaces with $\C$) and that for any prime power $p^k > 1$ we have
\[
    \chi_{p^k}(g) = \chi_{p^k}^\circ(g) + \chi_{p^{k-1}}(g).
\]
In turn, the last identity follows directly from \cite[Definition 4.5 and Lemma 4.3]{pascadi2025nonabelian}. The fact that $\chi_c(g)$ is the number of fixed points of $g$ in $\P^1(\Z/c\Z)$ is immediate from the fact that $\rho_c(g)$ is the permutation map associated to the action of $g$ on $\P^1(\Z/c\Z)$.

Finally, part $(iii)$ follows directly from \cite[Proposition 4.6]{pascadi2025nonabelian}.
\end{proof}

\section{Non-abelian and quadratic characters} \label{sec:non-abelian-quadratic}

\subsection{From Kloosterman sums to non-abelian characters}

In this subsection, we reduce Type-II bounds for bilinear forms with Kloosterman sums to estimating a certain sum involving characters of $\SL_2(\Z/c\Z)$, following \cite[\S 4]{pascadi2025nonabelian}. We also amplify, following \cite[\S 5]{pascadi2025nonabelian}, to further reduce to characters of $\SL_2(\Z/(c/d)\Z)$, for a suitable divisor $d \mid c$. We recall \cref{not:non-abelian-chars}, \cref{eq:nu-function}, and \cref{eq:generators}. 

\begin{proposition}[Fourier analysis $+$ Amplification \cite{pascadi2025nonabelian}] \label{prop:kloosterman-to-characters}
Let $c = c_1c_2$ where $c_1, c_2 \in \Z_+$, $c_1$ is square-free, $c_2$ is square-full, and $(c_1, c_2) = 1$. 
Let $a \in (\Z/c\Z)^\times$, $M, N \in \Z \cap [1, c]$, and $\mI, \mJ \subset \Z$ be intervals with $|\mI| = M$, $|\mJ| = N$. Let $\eps > 0$ and $H_1 := 2c^{1+\eps}M^{-1}$, $H_2 := 2c^{1+\eps} N^{-1}$. Then there exist absolutely-bounded complex numbers $z_1(h), \ldots, z_{2k}(h) \ll 1$ such that for any complex sequences $(\alpha_m)_{m \in \mI}$, $(\beta_n)_{n \in \mJ}$, and any $k \in \Z_+$, one has
\begin{equation}\label{eq:kloosterman-to-characters}
    \Big\vert \sum_{m \in \mI} \sum_{n \in \mJ} \alpha_m \beta_n S(am, n; c) \nu_{(m,n,c_1)} \one_{(m,n,c_2) = 1} \Big\vert
    \le
    \|\alpha\| \|\beta\|
    \Big(\frac{c^{1+2\eps}}{\sqrt{H_1H_2}} \msS^{\frac{1}{2k}} + O_\eps(c^{-100})\Big),
\end{equation}
where $\nu_{(m, n, c_1)}$ is given by \cref{eq:nu-function}, and
\begin{equation}\label{eq:S-sum}
    \msS := \sum_{\substack{h_1, \ldots, h_{2k} \in \Z \\ |h_i| \le H_j\, \forall i \equiv j \pmod{2} \\ g := (-1)^k T^{a_1h_1} S \cdots T^{a_{2k} h_{2k}}S}} z_1(h_1) \cdots z_{2k}(h_{2k})\, \chi_c^\circ(g),
    \qquad 
    a_i := \begin{cases}
        \bar{a}, &i \equiv 1 \pmod{2}, \\ 
        1, &i \equiv 0 \pmod{2},
    \end{cases}
\end{equation}
for some $\bar{a} \in \Z$ with $a\bar{a} \equiv 1 \pmod{c}$, where $\chi_c^\circ$ is as in \cref{not:non-abelian-chars}. Moreover, for any $d \in \Z_+$ with $d \mid c$ and $(d, \tfrac{c}{d}) = 1$, one has
\begin{equation}\label{eq:amplification}
    \msS \ll d^{2+o(1)} \sum_{\substack{h_1, \ldots, h_{2k} \in \Z \\ |h_i| \le H_j\,\forall i \equiv j \pmod{2} \\ g := (-1)^k T^{a_1h_1}S \cdots T^{a_{2k}h_{2k}}S}} z_1(h_1) \cdots z_{2k}(h_{2k})\, \one_{g \equiv I \pmod{d}}\, \chi_{c/d}^\circ(g).
\end{equation}
\end{proposition}

\begin{proof}
By the characterization of spectral norms from \cref{eq:spectral-norm}, we have
\begin{equation}\label{eq:bil-sum-sp-norm}
    \Big\vert \sum_{m \in \mI} \sum_{n \in \mJ} \alpha_m \beta_n S(am, n; c) \nu_{(m,n,c_1)} \one_{(m,n,c_2) = 1} \Big\vert 
    \le 
    \|\alpha\| \|\beta\| \|K\|,
\end{equation}
where $K \in \C^{\mI \times \mJ}$ be the $M \times N$ matrix, indexed by $m \in \mI$ and $n \in \mJ$, with entries given by
\[
    K_{m,n} := S(am, n; c) \nu_{(m,n,c_1)} \one_{(m,n,c_2)=1}.
\]
We recall the notation for the Fourier transform on finite groups from \cref{eq:Fourier-tr-nonab}.
By \cite[Corollary 4.11]{pascadi2025nonabelian} with $H_j \gets H_j/2$, we can bound
\begin{equation}\label{eq:pass-to-reps}
    \|K\| \le c^{1+2\eps} \|\hat{F}(\rho_c^\circ)\| + O_\eps(c^{-100}),
\end{equation}
where $\rho_c^\circ$ is as in \cref{not:non-abelian-chars}, and $F : \SL_2(\Z/c\Z) \to \C$ is given by
\[
    F := \frac{4}{H_1H_2} \sum_{\substack{|h_1| \le H_1/2 \\ |h_2| \le H_2/2}} w_1(h_1) w_2(h_2) \one_{T^{\bar{a}h_1} S T^{h_2}},
\]
for some absolutely-bounded complex numbers $w_1(h), w_2(h) \ll 1$.

Now let $\rho$ be any representation of $\SL_2(\Z/c\Z)$ and $\chi := \Tr\, \rho$. By \cref{eq:trace-method} we have
\begin{equation}\label{eq:apply-trace-method}
    \|\hat{F}(\rho)\|^{2k} \le \Tr\Big( \Big(\hat{F}(\rho) \hat{F}(\rho)^*\Big)^k \Big),
\end{equation}
where the trace can be expanded as
\[
\begin{aligned}
    \frac{4^{2k}}{(H_1 H_2)^{2k}}
    \sum_{\substack{h_1, h_1', \ldots, h_{2k}, h_{2k}' \in \Z \\ |h_i|, |h_i'| \le H_j/2\ \forall i \equiv j \pmod{2}}}
    w_1(h_1) w_2(h_2) \bar{w_2(h_2') w_1(h_3')} w_1(h_3) w_2(h_4) \cdots 
    \bar{w_2(h_{2k}') w_1(h_1')}
    \\
    \times\, 
    \chi \Big(T^{\bar{a}h_1}ST^{h_2}T^{-h_2'} S^{-1}T^{-\bar{a} h_3'} T^{\bar{a}h_3}ST^{h_4} \cdots T^{-h_{2k}'} S^{-1}T^{-\bar{a}h_1'} \Big).
\end{aligned}
\]
For any $i \in \{1, \ldots, 2k\}$ and $j \in \{1, 2\}$ such that $i \equiv j \pmod{2}$, and any $\tilde{h} \in \Z$, we define
\begin{equation}\label{eq:zi}
    z_i(\tilde h) := \frac{4}{H_j} \sum_{\substack{|h|, |h'| \le H_j/2 \\ h - h' = \tilde h}} w_j(h) \bar{w_j(h')}.
\end{equation}
Using the fact that $S^{-1} = -S$ and the notation from \cref{eq:S-sum,eq:zi}, we can rewrite the trace as
\begin{equation}\label{eq:trace-computation}
\begin{aligned}
    \Tr\Big( \Big(\hat{F}(\rho) \hat{F}(\rho)^*\Big)^k \Big)
    =
    \frac{1}{(H_1H_2)^k}
    \sum_{\substack{h_1, \ldots, h_{2k} \in \Z \\ |h_i| \le H_j\, \forall i \equiv j \pmod{2} \\ g := (-1)^kT^{a_1h_1}S\cdots T^{a_{2k}h_{2k}}S}} z_1(h_1) \cdots z_{2k}(h_{2k})\, \chi(g).
\end{aligned}
\end{equation}
Combining \cref{eq:bil-sum-sp-norm,eq:pass-to-reps,eq:apply-trace-method,eq:trace-computation} with $\rho = \rho_c^\circ$ completes the proof of \cref{eq:kloosterman-to-characters} (in particular, the sum $\msS$ from \cref{eq:S-sum} is nonnegative).

We now perform the amplification step; this is similar to \cite[Proposition 5.1]{pascadi2025nonabelian}, but significantly simpler due to the assumption that $(d, \tfrac{c}{d}) = 1$. By \cref{lem:non-abelian-chars}$(i)$, we have
\[
    \msS = \sum_{\substack{h_1, \ldots, h_{2k} \in \Z \\ |h_i| \le H_j\, \forall i \equiv j \pmod{2} \\ g := (-1)^k T^{a_1h_1} S \cdots T^{a_{2k} h_{2k}}S}} z_1(h_1) \cdots z_{2k}(h_{2k})\, \chi_d^\circ(g)\, \chi_{c/d}^\circ(g).
\]
By splitting $\chi_d^\circ$ into irreducible characters, we can write
\begin{equation}\label{eq:before-amplif}
    \msS = \sum_{\substack{\rho' \in \hat{\SL_2}(\Z/d\Z) \\ \chi' := \Tr(\rho')}} \Mult(\rho', \rho_d^\circ) 
    \sum_{\substack{h_1, \ldots, h_{2k} \in \Z \\ |h_i| \le H_j\, \forall i \equiv j \pmod{2} \\ g := (-1)^k T^{a_1h_1} S \cdots T^{a_{2k} h_{2k}}S}} z_1(h_1) \cdots z_{2k}(h_{2k})\, \chi'(g)\, \chi_{c/d}^\circ(g).
\end{equation}
The inner sum in \cref{eq:before-amplif} is nonnegative for any $\rho' \in \hat{\SL_2}(\Z/d\Z)$, due to \cref{eq:trace-computation} with $\rho = \rho' \otimes \rho_{c/d}^\circ$.
Moreover, by \cref{lem:non-abelian-chars}$(iii)$, we have $\Mult(\rho', \rho_d^\circ) \ll d^{o(1)}$, and whenever this multiplicity is nonzero, $\dim \rho' \gg d^{1-o(1)}$. If $R$ denotes the (left-)regular representation of $\SL_2(\Z/d\Z)$, it follows that 
\[
    \Mult(\rho', \rho_d^\circ) \ll d^{o(1)} \frac{\dim \rho'}{d}
    =
    d^{o(1)} \frac{\Mult(\rho', R)}{d}.
\]
We conclude from this, \cref{eq:before-amplif}, and \cref{eq:trace-dec} that
\[
\begin{aligned}
    \msS &\ll \frac{d^{o(1)}}{d} \sum_{\substack{\rho' \in \hat{\SL_2}(\Z/d\Z) \\ \chi' := \Tr(\rho')}} \Mult(\rho', R) 
    \sum_{\substack{h_1, \ldots, h_{2k} \in \Z \\ |h_i| \le H_j\, \forall i \equiv j \pmod{2} \\ g := (-1)^k T^{a_1h_1} S \cdots T^{a_{2k} h_{2k}}S}} z_1(h_1) \cdots z_{2k}(h_{2k})\, \chi'(g)\, \chi_{c/d}^\circ(g)
    \\
    &= \frac{d^{o(1)}}{d} 
    \sum_{\substack{h_1, \ldots, h_{2k} \in \Z \\ |h_i| \le H_j\, \forall i \equiv j \pmod{2} \\ g := (-1)^k T^{a_1h_1} S \cdots T^{a_{2k} h_{2k}}S}} z_1(h_1) \cdots z_{2k}(h_{2k})\, \Tr R(g)\, \chi_{c/d}^\circ(g).
\end{aligned}
\]
But $\Tr R(g)$ is simply $|\SL_2(\Z/d\Z)| \one_{g \equiv I \pmod{d}}$, which proves \cref{eq:amplification} in light of $|\SL_2(\Z/d\Z)| \asymp d^3$. In particular, the right-hand side of \cref{eq:amplification} is nonnegative.
\end{proof}

\subsection{Passing to quadratic characters}
A key input in this work is that when $c$ is odd and square-free, $\chi_c^\circ(g)$ can be directly related to the Jacobi symbol modulo $c$.
To formalize this, we need the following notation (which will only be used in this subsection).

\begin{notation}[Levels] \label{not:level}
We define the \emph{level} of $g \in \SL_2(\Z)$ by
\[
    \lev(g) := 
    \begin{cases}
        0, & g \in \{\pm I\}, \\ 
        \max \left\{d \in \Z_+ : g \in \ker\Big(\SL_2(\Z) \to \PSL_2(\Z/d\Z)\Big)\right\},
        & \text{otherwise.}
    \end{cases}
\]
\end{notation}

Note that by \cref{eq:center}, for $g \in \SL_2(\Z)$ and $d \in \Z_+$, we have the equivalences
\begin{equation}\label{eq:level-charact}
\begin{aligned}
    d \mid \lev(g)
    \quad &\iff \quad
    g \in \ker(\SL_2(\Z) \to \PSL_2(\Z/d\Z))
    \\
    &\iff \quad 
    \exists \gamma \in \Z/d\Z,\ \gamma^2 = 1 :
    g \equiv \gamma I \pmod{d}.
\end{aligned}
\end{equation}

\begin{lemma} \label{lem:conn-to-quadratic}
If $c$ is odd and square-free, then for any $g \in \SL_2(\Z)$, one has the explicit expression
\[
    \chi_c^\circ(g) = d \Big(\frac{\Tr(g)^2-4}{c/d}\Big),
    \qquad 
    \text{for } 
    d = (\lev(g), c).
\]
in terms of the Jacobi symbol from \cref{eq:jacobi-symbol}.
\end{lemma}

\begin{proof}
By the multiplicativity property in \cref{lem:non-abelian-chars}$(i)$, it suffices to show that for any prime $p$, one has
\begin{equation}\label{eq:conn-to-quadratic-prime}
    \chi_p^\circ(g) = 
    \begin{cases} 
    p, & g \equiv \pm I \pmod{p}, \\
    \Big(\frac{\Tr(g)^2-4}{p}\Big), & \text{otherwise.}
    \end{cases}
\end{equation}
By \cref{lem:non-abelian-chars}$(ii)$, $1 + \chi_p^\circ(g)$ is the number of fixed points of $g$ in the projective line $\P^1(\Z/p\Z)$. We write 
\[
    g = \begin{pmatrix}
    x & y \\ z & t
    \end{pmatrix}
    \in 
    \SL_2(\Z/p\Z),
\]
so that $xt - yz = 1$ in $\SL_2(\Z/p\Z)$. Note that $g$ fixes the point $[1 : 0]$ at infinity if and only if $z = 0$. Moreover, for $u \in \Z/p\Z$, $g$ fixes the point $u = [u : 1]$ if and only if $xu+y = (zu+t)u$ (this implies that $zu+t$ is invertible, since otherwise $zu+t = xu+y = 0$ forces $u = 0 = t = y$, contradicting the determinant condition).

\emph{Case 1:} Suppose $z \neq 0$. Then the quadratic equation
\[
    zu^2 + (t-x)u - y = 0
\]
has exactly $1 + \Big(\tfrac{(t-x)^2 + 4yz}{p}\Big)$ solutions $u \in \Z/p\Z$. But $(t-x)^2 + 4yz = (t+x)^2 - 4(xt-yz) = \Tr(g)^2 - 4$, so \cref{eq:conn-to-quadratic-prime} holds.

\emph{Case 2:} Suppose $z = 0$. Then the equation
\[
    (t-x)u - y = 0
\]
has $\one_{t-x \neq 0} + p\one_{t-x = y = 0}$ solutions $u \in \Z/p\Z$, to which we must add the point at infinity. This gives a total of 
\[
    1 + \one_{t-x \neq 0} + p\one_{t-x = y = 0}
\]
fixed points of $g$ in $\P^1(\Z/p\Z)$. If $g \equiv \pm I \pmod{p}$, then this is just $p+1$, so \cref{eq:conn-to-quadratic-prime} holds. If $g \not\equiv \pm I \pmod{p}$, then we cannot have $z = t-x = y = 0$, so we're left with
\[
    1 + \one_{t-x \neq 0}
    =
    1 + \Big(\tfrac{(t-x)^2}{p}\Big)
    =
    1 + \Big(\tfrac{(t-x)^2 + 4yz}{p}\Big)
\]
fixed points, and \cref{eq:conn-to-quadratic-prime} holds once again.
\end{proof}

\begin{proposition}[Passing to quadratic character sums] \label{prop:pass-to-quadratic}
Assuming the setup of \cref{prop:kloosterman-to-characters}, one has
\[
    \msS \ll c_2^{1+o(1)} \sum_{\substack{d \in \Z_+ \\ c_2 \mid d \mid c,\, 2d \nmid c}} d \sum_{\substack{\gamma \in \Z/d\Z \\ \gamma^2 = 1}}
    \Bigg\vert 
    \sum_{\substack{h_1, \ldots, h_{2k} \in \Z \\ |h_i| \le H_j\,\forall i \equiv j \pmod{2} \\ g := T^{a_1h_1}S \cdots T^{a_{2k}h_{2k}}S \\ g \equiv \gamma I \pmod{d}}} z_1(h_1) \cdots z_{2k}(h_{2k}) \Big(\frac{\Tr(g)^2-4}{c/d}\Big)
    \Bigg\vert.
\]
\end{proposition}

\begin{proof}
Let $\tilde c_2$ be either $c_2$ if $2 \nmid c_1$, or $2c_2$ if $2 \mid c_1$; in either case, $c/\tilde c_2$ is odd and square-free, and $c_2 \asymp \tilde c_2$. We apply \cref{eq:amplification} with $d \gets \tilde c_2$ to obtain
\[
    \msS \ll c_2^{2+o(1)} \sum_{\substack{h_1, \ldots, h_{2k} \in \Z \\ |h_i| \le H_j\,\forall i \equiv j \pmod{2} \\ g := (-1)^k T^{a_1h_1}S \cdots T^{a_{2k}h_{2k}}S}} z_1(h_1) \cdots z_{2k}(h_{2k})\, \one_{g \equiv I \pmod{\tilde c_2}}\, \chi_{c/\tilde c_2}^\circ(g).
\]
We then apply \cref{lem:conn-to-quadratic} with $c \gets c/\tilde c_2$ to obtain
\[
    \msS \ll c_2^{2+o(1)} \sum_{d \mid \frac{c}{\tilde c_2}} d \sum_{\substack{h_1, \ldots, h_{2k} \in \Z \\ |h_i| \le H_j\,\forall i \equiv j \pmod{2} \\ g := (-1)^k T^{a_1h_1}S \cdots T^{a_{2k}h_{2k}}S \\ (\lev(g), c/\tilde c_2) = d}} z_1(h_1) \cdots z_{2k}(h_{2k})\, \one_{g \equiv I \pmod{\tilde c_2}} \Big(\frac{\Tr(g)^2-4}{c/(d\tilde c_2)}\Big).
\]
Note that $(\tilde c_2, c/\tilde{c_2}) = 1$ since $c/\tilde{c_2}$ divides $c_1$, the square-free part of $c$. Moreover, the restriction $g \equiv I \pmod{\tilde c_2}$ implies $\tilde c_2 \mid \lev(g)$. In this case, we have
\[
    (\lev(g), c) = (\lev(g), \tilde c_2) \cdot \Big(\lev(g), \frac{c}{\tilde c_2}\Big) = 
    \tilde c_2 \cdot \Big(\lev(g), \frac{c}{\tilde c_2}\Big).
\]
By changing variables $d \gets d/\tilde c_2$ and the triangle inequality, we thus obtain
\[
    \msS \ll c_2^{1+o(1)} \sum_{\substack{d \in \Z_+ \\ \tilde c_2 \mid d \mid c}} d \Bigg\vert \sum_{\substack{h_1, \ldots, h_{2k} \in \Z \\ |h_i| \le H_j\,\forall i \equiv j \pmod{2} \\ g := (-1)^k T^{a_1h_1}S \cdots T^{a_{2k}h_{2k}}S \\ (\lev(g), c) = d}} z_1(h_1) \cdots z_{2k}(h_{2k})\, \one_{g \equiv I \pmod{\tilde c_2}} \Big(\frac{\Tr(g)^2-4}{c/d}\Big) \Bigg\vert.
\]
We claim that the condition $(\lev(g), c) = d$ in the inner sum can be replaced by $d \mid \lev(g)$. Indeed, if $c_2 \mid d \mid \lev(g)$ and $(\lev(g), c) = d' \neq d$, then $c/d'$ divides (but does not equal) the square-free number $c/d$. Thus there exists a prime $p$ which divides $c/d$ but not $c/d'$, so $p \mid d' \mid \lev(g)$. In this case, $\Tr(g)^2 \equiv 4 \pmod{p}$, and the Jacobi symbol vanishes. The condition $d \mid \lev(g)$ can then be expanded by \cref{eq:level-charact} (combined with the triangle inequality) to obtain
\[
    \msS \ll c^{o(1)} c_2 \sum_{\substack{d \in \Z_+ \\ \tilde c_2 \mid d \mid c}} d \sum_{\substack{\gamma \in \Z/d\Z \\ \gamma^2 = 1}} \Bigg\vert \sum_{\substack{h_1, \ldots, h_{2k} \in \Z \\ |h_i| \le H_j\,\forall i \equiv j \pmod{2} \\ g := (-1)^k T^{a_1h_1}S \cdots T^{a_{2k}h_{2k}}S \\ g \equiv \gamma I \pmod{d}}} z_1(h_1) \cdots z_{2k}(h_{2k})\, \one_{g \equiv I \pmod{\tilde c_2}} \Big(\frac{\Tr(g)^2-4}{c/d}\Big) \Bigg\vert.
\]
The condition that $g \equiv I \pmod{\tilde c_2}$ simply restricts the maximum to those $\gamma \equiv 1 \pmod{\tilde c_2}$, and can be ignored.
Finally, the condition that $\tilde c_2 \mid d$ is equivalent to $c_2 \mid d$ and $2d \nmid c$, and changing $g \mapsto (-1)^k g$, $\gamma \mapsto (-1)^k \gamma$ brings us to the desired bound.
\end{proof}

\subsection{Grouping variables}
We now aim to apply H\"older's inequality, to reduce bounding our sum of interest to two counting problems and a character sum with two complete variables. But before this, we need to discard certain terms corresponding to large values of $d$ in \cref{prop:pass-to-quadratic}; the following lemma borrowed from \cite{pascadi2025nonabelian} is helpful.

\begin{lemma}[\cite{pascadi2025nonabelian}] \label{lem:h14-bound}
Let $c \in \Z_+$, $\gamma \in \Z/c\Z$ with $\gamma^2 = 1$, $H_1, H_2 \ge 1$, and $a_1, a_2 \in \Z$ with $(a_1a_2, c) = 1$. Then 
\[
    \sum_{\substack{h_1, h_2, h_3, h_4 \in \Z \\ |h_1|, |h_3| \le H_1,\, |h_2|, |h_4| \le H_2 \\ T^{a_1h_1}S T^{a_2h_2} S T^{a_1h_3} S T^{a_2h_4}S \equiv \gamma I \\ \pmod{c}}}
    1
    \ll 
    c^{o(1)} \Big(1 + \frac{\min(H_1, H_2)^2}{c^2}\Big) \Big(1 + \frac{\max(H_1, H_2)}{c}\Big) \max(H_1, H_2).
\]
\end{lemma}
\begin{proof}
This follows immediately from \cite[Corollary A.2]{pascadi2025nonabelian}.
\end{proof}

\begin{proposition}[H\"older step] \label{prop:holder-step}
Assume the setup of \cref{prop:kloosterman-to-characters} with $2k = 4$. Then for any $\ell \in \Z_+$, one has
\[
    \msS \ll_\eps c^{2\eps} (cc_2 + H_1H_2)(H_1 + H_2) + c^\eps c_2 \sum_{\substack{d \in \Z_+ \\ c_2 \mid d \mid c,\, 2d \nmid c \\ d \le \min(H_1, H_2)}} d \sum_{\substack{\gamma \in \Z/d\Z \\ \gamma^2 = 1}} 
    \min\Big(\frac{H_1^2 H_2^2}{d^3}, \msA^{1-\frac{1}{\ell}} \msB^{\frac{1}{2\ell}} \msC_\ell^{\frac{1}{2\ell}}\Big),
\]
where
\begin{align}
    \msA = \msA(d,\gamma) &:= \sum_{r \pmod{d}} \sum_{\substack{h_1, h_2, h_3 \in \Z \\ h_1h_2h_3 \neq 0,\ h_1+h_3 \neq 0 \\ |h_1|, |h_3| \le H_1,\, |h_2| \le H_2 \\ T^{\bar{a}h_1}S T^{h_2}S T^{\bar{a}h_3}S T^{r}S \equiv \gamma I \pmod{d}}} 1,
    \label{eq:A-sum}
    \\ 
    \msB = \msB(d,\gamma) &:= \sum_{\substack{x, y \pmod{c/d} \\ r \pmod{d}}} \Bigg\vert \sum_{\substack{h_1, h_2, h_3 \in \Z \\ h_1h_2h_3 \neq 0,\ h_1+h_3 \neq 0 \\ |h_1|, |h_3| \le H_1,\, |h_2| \le H_2 \\ T^{\bar{a}h_1}S T^{h_2}S T^{\bar{a}h_3}S T^{r}S \equiv \gamma I \pmod{d} \\ \bar{a} h_1h_2h_3 - h_1 - h_3 \equiv x \pmod{c/d} \\ (h_1+h_3) h_2 \equiv y \pmod{c/d}}} 1 \Bigg\vert^2,
    \label{eq:B-sum}
    \\ 
    \msC_\ell = \msC_\ell(d,\gamma) &:= 
    \sum_{\substack{x, y \pmod{c/d} \\ r \pmod{d}}} \Bigg\vert
    \sum_{\substack{|h| \le H_2 \\ h \equiv r \pmod{d}}} z_4(h) \Big(\frac{(xh+y)^2-4}{c/d}\Big) \Bigg\vert^{2\ell}.
    \label{eq:C-sum}
\end{align}
\end{proposition}

\begin{proof}
We apply \cref{prop:pass-to-quadratic} with $2k = 4$. This gives
\begin{equation}\label{eq:S-to-T}
    \msS \ll c_2^{1+o(1)} \sum_{\substack{d \in \Z_+ \\ c_2 \mid d \mid c,\, 2d \nmid c}} d \sum_{\substack{\gamma \in \Z/d\Z \\ \gamma^2 = 1}}
     |\msT(d,\gamma)|,
\end{equation}
where
\[
    \msT(d,\gamma) := \sum_{\substack{h_1, h_2, h_3, h_4 \in \Z \\ |h_1|, |h_3| \le H_1,\, |h_2|, |h_4| \le H_2 \\ g := T^{a_1h_1}S \cdots T^{a_4h_4}S \\ g \equiv \gamma I \pmod{d}}} z_1(h_1) \cdots z_4(h_4) \Big(\frac{\Tr(g)^2-4}{c/d}\Big).
\]
Recall from \cref{prop:kloosterman-to-characters} that $a_1 = a_3 = \bar{a}$, $a_2 = a_4 = 1$, and $z_i(h_i) \ll 1$. So the triangle inequality and \cref{lem:h14-bound} give a trivial bound for $\msT$ of
\begin{equation}\label{eq:T-trivial-bound}
    \msT(d, \gamma) \ll d^{o(1)} \Big(1 + \frac{\min(H_1, H_2)^2}{d^2}\Big) \Big(1 + \frac{\max(H_1, H_2)}{d}\Big) \max(H_1, H_2).
\end{equation}
If $d > \min(H_1, H_2)$, the first parenthesis in \cref{eq:T-trivial-bound} can be ignored. Therefore, using the divisor bound and \cref{eq:center}, we have
\[
\begin{aligned}
    c_2 \sum_{\substack{d \in \Z_+ \\ c_2 \mid d \mid c,\, 2d \nmid c}} d \sum_{\substack{\gamma \in \Z/d\Z \\ \gamma^2 = 1}} |\msT(d,\gamma)|
    &\ll c^{o(1)} c_2 \max_{\substack{d \in \Z_+ \\ c_2 \mid d \mid c,\, 2d \nmid c}} d \Big(1 + \frac{\max(H_1, H_2)}{d}\Big) \max(H_1, H_2)
    \\ 
    &=
    c^{o(1)} c_2 \Big(c + \max(H_1, H_2)\Big) \max(H_1, H_2)
    \\
    &\ll_\eps
    c^{1+2\eps} c_2 \max(H_1, H_2),
\end{aligned}
\]
where we recalled that $\max(H_1, H_2) \ll c^{1+\eps}$ in the setup of \cref{prop:kloosterman-to-characters}. Combining this with \cref{eq:S-to-T} gives
\begin{equation}\label{eq:S-to-small-d}
    \msS \ll_\eps c^{1+2\eps}c_2 \max(H_1, H_2) + c_2^{1+o(1)} \sum_{\substack{d \in \Z_+ \\ c_2 \mid d \mid c,\, 2d \nmid c \\ d \le \min(H_1, H_2)}} d \sum_{\substack{\gamma \in \Z/d\Z \\ \gamma^2 = 1}} |\msT(d,\gamma)|.
\end{equation}
Henceforth, we consider $\msT(d,\gamma)$ in the range
\[
    d \le \min(H_1, H_2). 
\] 
Here, the trivial bound from \cref{eq:T-trivial-bound} reads
\begin{equation}\label{eq:T-trivial-bound-2}
    \msT(d,\gamma) \ll d^{o(1)} \frac{\min(H_1,H_2)^2 \max(H_1,H_2)^2}{d^3}
    =
    d^{o(1)} \frac{H_1^2 H_2^2}{d^3}.
\end{equation}
For later convenience, we aim to remove the contribution to $\msT(d,\gamma)$ of certain degenerate terms, which satisfy $h_1h_2h_3 = 0$ or $h_1 + h_3 = 0$. We have
\begin{equation}\label{eq:remove-0}
\begin{aligned}
    \msT(d, \gamma) = \msT'(d, \gamma) + O\Bigg(\sum_{\substack{h_1, h_2, h_3, h_4 \in \Z \\ h_1h_2h_3 = 0 \text{ or } h_1 + h_3 = 0 \\ |h_1|, |h_3| \le H_1,\, |h_2|, |h_4| \le H_2 \\ T^{\bar{a}h_1}ST^{h_2}ST^{\bar{a}h_3}ST^{h_4}S \equiv \gamma I \pmod{d}}} 1\Bigg),
\end{aligned}
\end{equation}
where
\begin{equation}\label{eq:T-prime-sum}
    \msT'(d,\gamma) := \sum_{\substack{h_1, h_2, h_3, h_4 \in \Z \\ h_1h_2h_3 \neq 0,\ h_1+h_3 \neq 0 \\ |h_1|, |h_3| \le H_1,\, |h_2|, |h_4| \le H_2 \\ g := T^{a_1h_1}S \cdots T^{a_4h_4}S \\ g \equiv \gamma I \pmod{d}}} z_1(h_1) \cdots z_4(h_4) \Big(\frac{\Tr(g)^2-4}{c/d}\Big).
\end{equation}
We note that the congruence $T^{\bar{a}h_1}S T^{h_2}S T^{\bar{a}h_3}S T^{h_4}S \equiv \gamma I \pmod{d}$ can be rewritten (after moving $(S T^{h_4}S)^{-1}$ to the right-hand side) as
\begin{equation}\label{eq:cong-mod-d}
    \begin{pmatrix}
        \bar{a}h_1h_2 - 1 & \bar{a}^2 h_1h_2h_3 - \bar{a}h_1 - \bar{a}h_3 \\ 
        h_2 & \bar{a}h_2h_3 - 1
    \end{pmatrix}
    \equiv
    -\gamma
    \begin{pmatrix}
        1 & 0 \\ 
        h_4 & 1
    \end{pmatrix}
    \pmod{d}.
\end{equation}
If $h_2 = 0$, then \cref{eq:cong-mod-d} forces $h_4 \equiv 0 \pmod{d}$ (due to the bottom-left entry) and then $h_1 \equiv -h_3 \pmod{d}$ (due to the top-right entry). Judging similarly for $h_1 = 0$ or $h_3 = 0$, and recalling that $d \le \min(H_1, H_2)$, we see that the terms with $h_1h_2h_3 = 0$ contribute to the error term in \cref{eq:remove-0} at most $O(\tfrac{H_2}{d} H_1 \tfrac{H_1}{d} + \tfrac{H_1}{d} H_2 \tfrac{H_2}{d})$.

If $h_1h_2h_3 \neq 0$ but $h_1 + h_3 = 0$, then \cref{eq:cong-mod-d} forces $-2 = -2\gamma \in \Z/d\Z$ (by adding the top-left and bottom-right entries). Multiplying the top-left entries by $2$, we then obtain $2\bar{a}h_1h_2 \equiv 0 \pmod{d}$, so $d \mid 2h_1h_2$. This gives $O( H_1H_2/d)$ ways to pick the nonzero integer $h_1h_2$, each leading to $c^{o(1)}$ ways to pick $h_1, h_2$. This fixes $h_3 = -h_1$, and we can pick $h_4$ subject to $h_2 \equiv - \gamma h_4 \pmod{d}$ in $O( H_2/d)$ ways. Overall, this case contributes to \cref{eq:remove-0} at most $O(c^{o(1)} \tfrac{H_1H_2}{d} \tfrac{H_2}{d})$. 

We conclude from \cref{eq:T-trivial-bound-2}, \cref{eq:remove-0}, and the analysis above that
\[
\begin{aligned}
    \msT(d, \gamma) &\ll \min \Big(\frac{H_1^2H_2^2}{d^3},  |\msT'(d, \gamma)| + c^{o(1)}\frac{H_1H_2(H_1+H_2)}{d^2}\Big),
    \\
    &\ll 
    c^{o(1)}\frac{H_1H_2(H_1+H_2)}{d^2} + \min \Big(\frac{H_1^2H_2^2}{d^3},  |\msT'(d, \gamma)|\Big).
\end{aligned}
\]
Combining this with \cref{eq:S-to-small-d}, \cref{eq:center}, and $d \ge c_2$ gives
\begin{equation}\label{eq:S-to-T-prime}
\begin{aligned}
    \msS \ll_\eps c^{1+2\eps}c_2 \max(H_1, H_2) &+ c^{o(1)}H_1H_2(H_1+H_2) 
    \\
    &+ c_2^{1+o(1)} \sum_{\substack{d \in \Z_+ \\ c_2 \mid d \mid c,\, 2d \nmid c \\ d \le \min(H_1, H_2)}} d \sum_{\substack{\gamma \in \Z/d\Z \\ \gamma^2 = 1}} \min\Big(\frac{H_1^2H_2^2}{d^3}, |\msT'(d,\gamma)|\Big).
\end{aligned}
\end{equation}

Finally, note that when $g = T^{\bar{a}h_1}S T^{h_2}S T^{\bar{a}h_3}S T^{h_4}S$, we have $\Tr(g) = xh_4 + y$ where
\[
    x = \bar{a}^2 h_1h_2h_3 - \bar{a}(h_1+h_3),
    \qquad 
    y = 2 - \bar{a} (h_1+h_3) h_2.
\]
By the triangle inequality in \cref{eq:T-prime-sum} and the bound $z_i(h_i) \ll 1$, we thus have
\[
    \msT'(d,\gamma) \ll \sum_{\substack{x, y \pmod{c/d} \\ r \pmod{d}}} \sum_{\substack{h_1, h_2, h_3 \in \Z \\ h_1h_2h_3 \neq 0,\ h_1+h_3 \neq 0 \\ |h_1|, |h_3| \le H_1,\, |h_2| \le H_2 \\ T^{\bar{a}h_1}S T^{h_2}S T^{\bar{a}h_3}S T^{r}S \equiv \gamma I \pmod{d} \\ \bar{a}^2 h_1h_2h_3 - \bar{a}(h_1+h_3) \equiv x \pmod{c/d} \\ \bar{a} (h_1+h_3) h_2 \equiv 2-y \pmod{c/d}}} \Big\vert
    \sum_{\substack{|h_4| \le H_2 \\ h_4 \equiv r \pmod{d}}} z_4(h_4) \Big(\frac{(xh_4+y)^2-4}{c/d}\Big) \Big\vert.
\]
We then apply H\"older's inequality to the sum over $x, y, r$, with parameters $\frac{\ell}{\ell-1}$, $2\ell$, and $2\ell$, to obtain
\[
    \msT'(d,\gamma) \ll \msA(d,\gamma)^{1-\frac{1}{\ell}} \msB(d,\gamma)^{\frac{1}{2\ell}}
    \msC_\ell(d,\gamma)^{\frac{1}{2\ell}},
\]
for $\msA, \msB, \msC_\ell$ as in \cref{eq:A-sum,eq:B-sum,eq:C-sum} (we execute the sum over $x, y$ for $\msA$, and change variables $x \mapsto ax$, $y \mapsto a(2-y)$ to slightly simplify $\msB$). By combining this with \cref{eq:S-to-T-prime}, we are done.
\end{proof}

\section{Counting and character sums} \label{sec:counting-problems}


Here we bound the sums $\msA, \msB$ and $\msC_{\ell}$ from \cref{eq:A-sum,eq:B-sum,eq:C-sum} by various counting techniques as well as estimating character and exponential sums. We note that the congruence $T^{\bar{a}h_1}S T^{h_2}S T^{\bar{a}h_3}S T^{r}S \equiv \gamma I \pmod{d}$ can be rewritten (by the same computation as in \cref{eq:cong-mod-d}) as
\[
    \begin{pmatrix}
        \bar{a}h_1h_2 - 1 & \bar{a}^2 h_1h_2h_3 - \bar{a}h_1 - \bar{a}h_3 \\ 
        h_2 & \bar{a}h_2h_3 - 1
    \end{pmatrix}
    \equiv
    -\gamma
    \begin{pmatrix}
        1 & 0 \\ 
        r & 1
    \end{pmatrix}
    \pmod{d}.
\]

\subsection{Bounding \texorpdfstring{$\msA$}{A}}

\begin{proposition} \label{prop:A-sum-bound}
Let $d \in \Z_+$, $a \in \Z$ such that $(a, d) = 1$, $\gamma \in \Z/d\Z$ such that $\gamma^2 = 1$ and $d \leq \min(H_1, H_2)$. Then the sum $\msA$ given by \cref{eq:A-sum} is bounded by
\[
    \msA \ll d^{o(1)} \frac{H_1^2 H_2}{d^2}.
\]
\end{proposition}

\begin{proof}
After summing over $r$, we see that $\msA$ is the number of triples $(h_1, h_2, h_3) \in [-H_1, H_1] \times [-H_2, H_2] \times [-H_1, H_1]$ such that $h_1 + h_3 \not= 0 \not = h_1h_2h_3$ and
$$\bar{a} h_1h_2 - 1 \equiv  -\gamma, \quad \bar{a}h_1h_2h_3 - h_1 - h_3 \equiv 0, \quad \bar{a} h_2h_3 - 1 \equiv - \gamma$$
modulo $d$. Let $\delta = (d, h_2)$. Then the first congruence implies that $\gamma \equiv 1$ (mod $\delta$), and so $h_1 \equiv \frac{1 - \gamma}{\delta} a \overline{\frac{h_2}{\delta}}$ (mod $\frac{d}{\delta}$) is determined modulo $d/\delta$. Substituting the third congruence into the second implies $-h_1 \gamma - h_3 \equiv 0$  (mod $d$). Since $h_1h_2h_3 \not= 0$, we get, for a given $\delta$, at most $O(H_2/\delta)$ values for $h_2$ and $O(H_1\delta/d)$ values for $h_1$ and then $O(H_1/d)$ values for $h_3$, and the claim follows, observing that the number of possible $\delta$ is $d^{o(1)}$.
\end{proof}

\subsection{Bounding \texorpdfstring{$\msB$}{B}}

\begin{proposition} \label{prop:B-sum-bound}
Let $c, d \in \Z_+$ such that $d \mid c$, $(d, c/d) = 1$, and $c/d$ is square-free. Let $a \in \Z$ such that $(a, c) = 1$, $\gamma \in \Z/d\Z$ such that $\gamma^2 = 1$, and $H_1, H_2$ such that $d \leq \min(H_1, H_2)$. Then the sum $\msB$ given by \cref{eq:B-sum} is bounded by
\[
    \msB \ll (H_1H_2c)^{o(1)} \frac{H_1^2H_2}{d^2} \Big(1 +   \frac{H_1H_2}{c}\Big) \Big(1 + \frac{H_1d}{c}\Big).
\]
\end{proposition}

\begin{proof}
As before, we need to count 6-tuples $(h_1, h_2, h_3, h_1', h_2', h_3')$ with $h_1h_2h_3h_1'h_2'h_3' \not = 0 \not= (h_1 + h_3)(h_1'+h_3')$ and $h_1, h_3, h_1', h_3' \ll H_1$, $h_2, h_2' \ll H_2$ such that
\begin{equation}\label{cong1}
(h_1 + h_3)h_2 \equiv (h_1' + h_3')h_2', \quad \bar{a}h_1h_2h_3 - h_1 - h_3 \equiv \bar{a} h_1'h_2'h_3' - h_1' - h_3' \, (\text{mod } c/d)
\end{equation}
and
\begin{equation}\label{cong2}
   \bar{a} h_1h_2 - 1 \equiv  -\gamma, \quad \bar{a}h_1h_2h_3 - h_1 - h_3 \equiv 0, \quad \bar{a} h_2h_3 - 1 \equiv - \gamma \, (\text{mod } d) ,
\end{equation}
\begin{equation}\label{cong3}
    \bar{a} h'_1h'_2 - 1 \equiv  -\gamma, \quad \bar{a}h'_1h'_2h'_3 - h'_1 - h'_3 \equiv 0, \quad \bar{a} h'_2h'_3 - 1 \equiv - \gamma, \quad h_2   \equiv h_2' \, (\text{mod } d).
\end{equation}
We write $g = (h_2, c/d)$, $g' = (h_2', c/d)$, and consider the subcount for a fixed pair $(g, g')$. By symmetry we can assume $g \geq g'$. 

We start by choosing nonzero $h_1, h_2, h_3$ in $O(d^{o(1)} H_1^2H_2/d^2g)$ ways using \eqref{cong2} by the same argument as in the previous proof together with the fact that in addition $g \mid h_2 \, (\not= 0)$.

The congruences in \eqref{cong3} imply that 
$(h_1' + h_3')h_2' \equiv 
2a(1 - \gamma) \, (\text{mod } d).$ 
Hence the nonzero number $(h_1' + h_3')h_2'$ is determined modulo $d$  and modulo $c/d$ by \eqref{cong1} and hence modulo $c$. Thus we have $O(1 + H_1H_2/c)$ choices for it, and so by a divisor argument $O((H_1H_2)^{o(1)} (1 + H_1H_2/c))$ choices for the pair $h' := (h_1'+h_3', h_2')$. 

Finally, the second congruence in \eqref{cong1} becomes
$$\bar{a} h_1' (h' - h_1') h_2' - h' \equiv \bar{a}h_1h_2h_3 - h_1 -h_3\, (\text{mod } c/d),$$
which is a quadratic congruence in $h_1'$ modulo the squarefree number $c/d$. It has at most $O(c^{o(1)} g')$ solutions.  Hence the total count for a fixed pair $(g, g')$ with $g \geq g'$ is 
$$(H_1H_2c)^{o(1)} \frac{H_1^2H_2}{d^2 g} \Big(1 + \frac{H_1H_2}{c}\Big) g'\Big(1 + \frac{H_1d }{c}\Big),$$
and the claim follows.
\end{proof}

The following variation using a Fourier-analytic argument performs better in certain ranges. 

\begin{proposition} \label{prop:B-sum-bound-alt}
Under the same assumption as in the previous proposition we have 
\[
    \msB \ll (H_1H_2c)^{o(1)} \frac{H_1H_2}{d} \Bigg( \Big(   \frac{c^{1/2}}{d^{1/2}  }  + \frac{H_1^2}{c } \Big)\Big(1 + \frac{H_1H_2}{c}\Big) + \frac{H_1^2 H_2^{1/2}}{cd^{1/4}}  \Bigg).
\]
\end{proposition}

\begin{proof}
It will be convenient to work with the variables
$$s := h_1 + h_3, \quad t := h_1 - h_3, \quad h := h_2$$
and similarly $s' := h_1' + h_3'$, $t' := h_1' - h_3'$, $h' := h_2'$. 

We start similarly as before and write $g = (h, c/d)$, $g' = (h', c/d)$. If $\msB(g, g')$ denotes the corresponding subcount, then by Cauchy-Schwarz we have $ \msB(g, g')^2 \leq \msB(g, g) \msB(g', g')$, so without loss of generality we can assume $g = g'$.  Note that under this assumption $g \mid s-s'$ from the middle congruence in \eqref{cong1}.

Let $G_1, G_2, G_3$ be three disjoint sets  of primes dividing  $c/(dg)$, and consider the subcount of $6$-tuples where
\begin{displaymath}
\begin{split}
& G_1  = \{p \mid c/(dg) : p \mid (hs - 4a, h's'-4a)\}, \\
&G_2  = \{p \mid c/(dg) : p \mid (s, s'), p \nmid hh', p \not\in G_1\}, \\
& G_3  = \{p \mid c/(dg) : p \nmid ss'hh', p \mid (s - s', h- h'), p \not\in G_1 \cup G_2\},
\end{split}
\end{displaymath}
and in addition $\delta = (h, d) = (h', d)$ (by the last congruence in \eqref{cong3}). There are $c^{o(1)}$ such subcounts. 

We write
$$g_1 = \prod_{p\in G_1} p, \quad 
g_2 = \prod_{p\in G_2} p, \quad g_3 = \prod_{p\in G_3} p.$$

It follows from \eqref{cong2} that $\bar{a}hs\equiv 2(1-\gamma) $ (mod $d$), and also $s \equiv 0$ (mod $\delta$), so $s$ is determined modulo $[d/\delta, \delta]$. The same holds for $s'$. After these preparations we start with the counting procedure. 

As a first step we choose the product $hs\not=0$, which is divisible by $g$ and $g_2$ and determined modulo $g_1$ (these three moduli are pairwise coprime). Moreover, $hs$ is divisible by $\delta^2$ and determined modulo   $d$. The product $hs$ determines the pair $(h, s)$ up to a divisor function, so that the number of choices of $(h, s)$ is at most
\begin{equation}\label{pair1}
    O\Big((H_1H_2)^{o(1)} \Big(1 + \frac{H_1H_2}{gg_1g_2 [d, \delta^2]}\Big)\Big).
\end{equation}
Before we continue, we remember that
\begin{equation}\label{bound}
    g \delta \ll H_2
\end{equation}
otherwise there is no solution at all. 

In the second step we choose 
the product $h's' \not= 0$ satisfying a congruence modulo $c/d$ and modulo $d$ in $O(1+H_1H_2/c)$, which again determines the pair $(h', s')$ up to a divisor function. Alternatively, we can choose $h'$, $s'$ separately where $h'$ is divisible by $g$ and satisfies a congruence modulo $g_3$ and modulo $d$ (by \eqref{cong3}), while $s'$ is divisible by $g_2$ and satisfies a congruence modulo $g_3$, $g_1$, $g$ and $[d/\delta, \delta]$ (once $h', s'$ are chosen). Hence in total the number of choices for this pair is
\begin{equation}\label{pair2}
O\Bigg((H_1H_2)^{o(1)}\min\Big( 1 + \frac{H_1H_2}{c}, \Big(1 + \frac{H_2}{gg_3d}\Big)\Big(1 + \frac{H_1}{gg_1g_2g_3[d/\delta, \delta]}\Big)\Big)\Bigg).
\end{equation}

In the final step we choose the pair $(t, t')$, both entries of which are divisible by $d' := d/\delta$, so let us write $t = d'\tau$, $t' = d'\tau'$.  Let $f$ be the odd part of $c/(dg).$  We will only use the second congruence in \eqref{cong1}, 
which after dividing out by $g$ and if necessary dropping the information modulo 2 becomes 
$$\tilde{h}\tau^2 - \tilde{h}'\tau'^2 \equiv Z \, (\text{mod } f), \quad Z =  \overline{d'}^2\frac{(s- s')}{g}(hs - 4a), \quad \tilde{h} = \frac{h}{g}, \quad \tilde{h}' = \frac{\tilde{h}'}{g},$$
where we used that $hs - 4a \equiv h's'- 4a$ (mod $c/d$). 

Before we proceed, we observe that   $(\tilde{h}\tilde{h}', f) = 1$ by construction. Moreover, for a prime $p \mid f$ we have $p\mid Z$ only if $p \mid g_1g_2g_3$. 

Let $W$ be a smooth non-negative function with support in $[-3, 3]$ that is one on $[-2, 2]$. Then our count for the pair $(t, t')$ is majorized by 
$$\sum_{\tau, \tau'} W\Big(\frac{\tau}{H_1/d'}\Big)W\Big(\frac{\tau'}{H_1/d'}\Big) \frac{1}{f} \sum_{b \, (\text{mod } f)} e\Big( \frac{b(\tilde{h}  \tau^2 - \tilde{h} '\tau'^2 - Z)}{f}\Big). $$
By Poisson summation, this equals 
\begin{equation}\label{poisson}
    \frac{H_1^2}{d'^2f^3} \sum_{z, z'} \hat{W}\Big(\frac{z}{d'f/H_1}\Big) \hat{W}\Big(\frac{z'}{d'f/H_1}\Big)\sum_{b, \eta, \eta' \, (\text{mod } f)} e\Big( \frac{b(\tilde{h}\eta ^2 - \tilde{h}' \eta'^2 - Z) - \eta z - \eta' z'}{f}\Big)
\end{equation}
where $\hat{W}(x) \ll_A (1 + |x|)^{-A}$ for any $A > 0$. Since $f$ is odd and squarefree, it suffices to estimate  the triple character sum for $f$ an odd  prime $p$. 

We first deal with the portion where $b \equiv 0$ (mod $p$), which equals $p^2 \one_{p \mid (z, z')}$. We continue with the remaining portion where $p \nmid b$.  We start with the $\eta$-sum, which is a standard quadratic Gau{\ss}  sum, which equals 
$$p^{1/2} \epsilon_p \Big(\frac{b\tilde{h}}{p}\Big) e\Big(\frac{- \overline{4b\tilde{h}} z^2}{p}\Big)$$
with $\epsilon_p = 1, i$ depending on whether $p \equiv 1$ (mod 4) or $p \equiv 3 $ (mod 4).  The same evaluation holds for the $\eta'$-sum, so that the complete $b, \eta, \eta'$-sum to an odd  prime modulus $p$ equals
\begin{displaymath}
\begin{split}
& p^2 \one_{p \mid (z, z')} + \sum_{0 \not\equiv b\, (\text{mod } p)} e\Big(\frac{-bZ}{p} \Big)  p \epsilon_p^2 \Big(\frac{-\tilde{h}\tilde{h}'}{p}\Big) e\Big(\frac{- \overline{4b} (\overline{\tilde{h}} z^2 - \overline{\tilde{h}'} z'^2) }{p}\Big) \\
& = p^2 \one_{p \mid (z, z')} +   
p \Big(\frac{\tilde{h}\tilde{h}'}{p}\Big) S( - Z, - \bar{4} (\overline{\tilde{h}} z^2 - \overline{\tilde{h}'} z'^2), p)\\
& \ll p^{3/2} [(p, z, z'), (p, g_1g_2g_3, z^2\tilde{h}' - z'^2 \tilde{h})]^{1/2}
\end{split}
\end{displaymath}
by Weil's bound for Kloosterman sums. We return to \eqref{poisson} and bound this expression by
\begin{displaymath}
\begin{split}
& f^{o(1)} \frac{H^2_1}{d'^2f^{3/2}}\sum_{z, z'}\Big| \hat{W}\Big(\frac{z}{d'f/H_1}\Big) \hat{W}\Big(\frac{z'}{d'f/H_1}\Big) \Big|[(f, z, z'), (g_1g_2g_3, z^2\tilde{h}' - z'^2 \tilde{h})]^{1/2}\\
&  \leq f^{o(1)} \frac{H^2_1}{d'^2f^{3/2}} \sum_{k_1 \mid f} \sum_{k_2 \mid g_1g_2g_3} \sum_{  z^2 \tilde{h}' \equiv   z'^2 \tilde{h} \, (\text{mod } \frac{k_2}{(k_2, k_1)})} \Big| \hat{W}\Big(\frac{z}{d'f  /k_1H_1}\Big) \hat{W}\Big(\frac{z'}{d'f /k_1H_1}\Big)  \Big| [k_1, k_2]^{1/2}\\
& \ll  f^{o(1)} \frac{H^2_1}{d'^2f^{3/2}} \sum_{k_1 \mid f} \sum_{k_2 \mid g_1g_2g_3} [k_1, k_2]^{1/2}\Big(1 +  \frac{fd'}{k_1H_1}\Big) \Big(1 +  \frac{fd'}{[k_1, k_2]H_1}\Big) \\
& \ll f^{o(1)} \frac{H^2_1}{d'^2f^{3/2}}\Big(f^{1/2} + \frac{ fd'(g_1g_2g_3)^{1/2}}{H_1} + \frac{d'^2f^2}{H_1^2}\Big)  =  f^{o(1)}\Big( \frac{H_1^2}{d'^2f} + \frac{H_1(g_1g_2g_3)^{1/2}}{d'f^{1/2}} + f^{1/2}\Big). 
\end{split}
\end{displaymath}
Combining this with \eqref{pair1} and \eqref{pair2}, we obtain the total count
\begin{displaymath}
\begin{split}
&  (cH_1H_2)^{o(1)}    \Big(1 + \frac{H_1H_2}{gg_1g_2[d, \delta^2]}\Big) \Big( \frac{H_1^2\delta^2}{d^2(c/dg)} + \frac{H_1\delta(g_1g_2g_3)^{1/2}}{d(c/dg)^{1/2}} + \frac{c^{1/2}}{(dg)^{1/2}}\Big) \\
& \quad\quad\quad \times \min\Big[ 1 + \frac{H_1H_2}{c}, \Big(1 + \frac{H_2}{gg_3d}\Big)\Big(1 + \frac{H_1}{gg_1g_2g_3[d/\delta, \delta]}\Big)\Big]. 
\end{split}
\end{displaymath}
Multiplying out, we obtain
\begin{displaymath}
\begin{split}
&  \Big(\frac{H_1^2g\delta^2 }{ cd} + \frac{H_1\delta(gg_1g_2g_3)^{1/2}}{(cd)^{1/2}} + \frac{c^{1/2}}{(dg)^{1/2}}+ \frac{H_1^3H_2\delta^2 }{d cg_1g_2[d, \delta^2]} + \frac{H_1^2H_2\delta g_3^{1/2}}{(gg_1g_2cd)^{1/2}[d, \delta^2]} + \frac{H_1H_2c^{1/2}}{d^{1/2}g^{3/2}g_1g_2[d, \delta^2]}\Big)\\ 
& \quad\quad\quad \times \min\Big[ 1 + \frac{H_1H_2}{c},  1 + \frac{H_2}{g_3d} + \frac{H_1}{g_3[d/\delta, \delta]} + \frac{H_1H_2}{g_3^2 d[d/\delta, \delta]} \Big]
(cH_1H_2)^{o(1)}.
\end{split}
\end{displaymath}
Using  $gg_1g_2g_3 \leq c/d$, $\delta \leq d \leq \min(H_1, H_2, c)$ 
and \eqref{bound}, we can bound all but the fifth term in the first line of the previous display by
$$\frac{c^{1/2}H_1H_2}{d^{3/2}} + \frac{H_1^3H_2}{cd}.$$
Moreover, we can estimate the minimum in the second line by
$$1 + \frac{H_2H_1^{1/2}}{(cdg_3)^{1/2}} + \frac{H_2^{1/2}H_1}{(g_3 c [d/\delta, \delta])^{1/2}} + \frac{H_1H_2}{c^{3/4} g_3^{1/2} (d[d/\delta, \delta])^{1/4}}.$$
Thus we estimate the term in question  by 
\begin{displaymath}
\begin{split}
&  (cH_1H_2)^{o(1)}\Bigg( \Big(   \frac{c^{1/2}H_1H_2}{d^{3/2}} + \frac{H_1^3H_2}{cd}\Big)\Big(1 + \frac{H_1H_2}{c}\Big) \\
& \quad\quad\quad + \frac{  H_1^2H_2\delta g_3^{1/2}}{(cd )^{1/2}[d, \delta^2] } \Big( 1 + \frac{H_2H_1^{1/2}}{(cdg_3)^{1/2}} + \frac{H_2^{1/2}H_1}{(cg_3  [d/\delta, \delta])^{1/2}} + \frac{H_1H_2}{c^{3/4} g_3^{1/2} (d[d/\delta, \delta])^{1/4}} \Big)\Bigg)\\ 
& \ll  (H_1H_2)^{o(1)}\Bigg( \Big(   \frac{c^{1/2} H_1H_2}{d^{3/2}  }  + \frac{H_1^3H_2}{cd } \Big)\Big(1 + \frac{H_1H_2}{c}\Big) + \frac{H_1^2 H_2}{c^{1/2}d} + \frac{H_1^{5/2} H_2^{2}}{cd^{3/2}} + \frac{H_1^{3} H_2^{3/2]}}{c d^{5/4}} 
+ \frac{H_1^3H_2^2 }{c^{5/4} d^{11/8}} 
\Bigg),
\end{split}
\end{displaymath}
using that $$\frac{\delta}{d^{1/2} [d, \delta^2][d/\delta, \delta]^{1/2}} \leq \frac{1}{d^{5/4}}, \quad \frac{\delta}{d^{1/2} [d, \delta^2]d^{1/4}[d/\delta, \delta]^{1/4}}   \leq \frac{1}{d^{11/8}}.  $$
Since we have
\begin{displaymath}
    \begin{split}
&\Big(\frac{c^{1/2} H_1H_2}{d^{3/2}} \cdot \frac{H_1^3H_2}{cd }  \Big)^{1/2}  \Big(1 + \frac{H_1H_2}{c} \Big) = \frac{H_1^2H_2}{c^{1/4}d^{5/4}}  + \frac{H_1^3H_2^2}{c^{5/4}d^{5/4}}  \geq \frac{H_1^2H_2}{c^{1/2}d} + \frac{H_1^3H_2^2}{c^{5/4}d^{11/8}}, \\&\Big(\frac{c^{1/2} H_1H_2}{d^{3/2}} \Big)^{3/4} \Big( \frac{H_1^3H_2}{cd }  \Big)^{1/4} \frac{H_1H_2}{c} = \frac{H_1^{5/2}H_2^2}{c^{7/8}d^{11/8}} \geq \frac{H_1^{5/2}H_2^2}{cd^{3/2}} ,\\
\end{split}
\end{displaymath}
we can drop three of the last four terms. This completes the proof. 
\end{proof}

\subsection{Bounding \texorpdfstring{$\msC_{2}$}{C2}}

We start with the Weil bound for quadratic character sums. 

\begin{lemma} \label{lem:char-sum-bound}
Let $p$ be an odd prime and $n \in \Z_+$. Let $f \in \F_p[X]$ be a degree-$d$ polynomial which is not a square. Then one has
\[
    \Big\vert \sum_{x \pmod{p}} \Big(\frac{f(x)}{p}\Big) \Big\vert
    \le 
    (d-1)
    \sqrt{p}.
\]
\end{lemma}

\begin{proof}
This is \cite[Theorem 2C', p.\,43]{schmidt1976bilinear}, see also \cite[Theorem 11.23, p.\,289]{iwaniec2004analytic}.
\end{proof}

\begin{proposition} \label{prop:C-sum-bound}
Let $c, d \in \Z_+$ such that $d \mid c$, $(d, c/d) = 1$, and $c/d$ is odd square-free. Let $\eps > 0$, and   $d\leq H_2 \ll c^{1+\eps}$. Let $z_4(h) \in \C$ satisfy $|z_4(h)| \ll 1$ for $|h| \le H_2$. Then the sum $\msC_2$ given by \cref{eq:C-sum} with $\ell = 2$ is bounded by
\begin{equation}\label{eq:C-sum-bound}
    \msC_2 \ll_\eps c^{O(\eps)} \Bigg( \frac{c^2}{d} \Big(\frac{H_2}{d}\Big)^2 + c \Big(\frac{H_2}{d}\Big)^4 \Bigg).
\end{equation}
\end{proposition}

\begin{remark}
This bound is essentially optimal. Indeed, expanding the product from \cref{eq:C-sum} gives $4$ copies of the $h$-variable; the first term in \cref{eq:C-sum-bound} corresponds to the `diagonal' tuples $(h_1, h_2, h_3, h_4)$ where the variables $h_j$ pair up, while the second term matches the `generic' tuples $(h_1, h_2, h_3, h_4)$, for which one expects square-root cancellation in the complete variables $x, y \pmod{\tfrac{c}{d}}$. Since $\tfrac{c}{d}$ can be composite, there may be many terms interpolating between the diagonal and the generic ones.
\end{remark}

\begin{proof}[Proof of \cref{prop:C-sum-bound}]
We write $H := H_2$ to slightly simplify the notation.
We expand \cref{eq:C-sum} with $\ell = 2$, swap sums, and use the assumption that $z_4(h) \ll 1$ to obtain
\begin{equation}\label{eq:C2-swap-sums}
    \msC_2 \ll \sum_{\substack{|h_1|, \ldots, |h_4| \le H \\ h_1 \equiv \cdots \equiv h_4 \pmod{d}}} \Bigg\vert 
    \underbrace{\sum_{x, y \pmod{c/d}} \prod_{j=1}^4 \Big(\frac{(xh_j+y-2)(xh_j+y+2)}{c/d}\Big)}_{=:E_{c/d}(h_1,h_2,h_3,h_4)} \Bigg\vert.
\end{equation}
By the multiplicativity of the Jacobi symbol, the Chinese remainder theorem, and the fact that $\tfrac{c}{d}$ is square-free, we have
\begin{equation}\label{eq:E-product}
    E_{c/d}(h_1,h_2,h_3,h_4) = \prod_{\text{prime } p \mid \frac{c}{d}} E_p(h_1,h_2,h_3,h_4).
\end{equation}
Let $p \mid \tfrac{c}{d}$ be a prime (recall that $\tfrac{c}{d}$ is odd, so $p$ is odd). Note that the pairs $(x, y) \in \F_p^2$ such that $x \neq 0$ are in a bijection with the pairs $(u, v) \in \F_p^2$ such that $u \neq v$, by the change of variables
\[
    (u, v) = \Big((y-2)\bar{x}, (y+2)\bar{x}\Big)
    \qquad 
    \iff 
    \qquad 
    (x, y) = \Big(4\bar{(v-u)}, 2(u+v)\bar{(v-u)}\Big).
\]
Of course, $E_p(h_1,h_2,h_3,h_4)$ includes the terms with $x = 0$. By separating these terms and then using the change of variables above, we get
\begin{equation}\label{eq:Ep-expansion}
\begin{aligned}
    E_p(h_1,h_2,h_3,h_4) &= \sum_{\substack{x, y \pmod{p} \\ x \neq 0}} \prod_{j=1}^4 \Big(\frac{(h_j+(y-2)\bar{x})(h_j+(y+2)\bar{x})}{p}\Big)
    + 
    \sum_{y \pmod{p}} \Big(\frac{(y-2)(y+2)}{p}\Big)^4
    \\
    &=
    \sum_{\substack{u, v \pmod{p} \\ u \neq v}} \prod_{j=1}^4 \Big(\frac{(h_j+u)(h_j+v)}{p}\Big)
    +
    (p-2)
    \\
    &= 
    \sum_{u, v \pmod{p}} \prod_{j=1}^4 \Big(\frac{(h_j+u)(h_j+v)}{p}\Big)
    + O(1)
    =
    \Big\vert \sum_{u \pmod{p}} \prod_{j=1}^4 \Big(\frac{h_j+u}{p}\Big)\Big\vert^2 + O(1).
\end{aligned}
\end{equation}

Now for each tuple $(h_1, h_2, h_3, h_4)$, we distinguish four types of primes $p \mid \tfrac{c}{d}$:
\begin{itemize}
    \item[I.] $p$ is such that $p \mid h_1 - h_2$ and $p \mid h_3 - h_4$. 
    \item[II.] $p$ is such that $p \mid h_1 - h_3$ and $p \mid h_2 - h_4$ (but $p$ is not of Type I).
    \item[III.] $p$ is such that $p \mid h_1 - h_4$ and $p \mid h_2 - h_3$ (but $p$ is not of Type I or II).
    \item[IV.] $p$ is such that $h_1, h_2, h_3, h_4 \pmod{p}$ cannot be arranged into two pairs of equal residues modulo $p$. Then the polynomial $f(X) = (X+h_1)(X+h_2)(X+h_3)(X+h_4)$ is not a square in $\F_p[X]$, and applying \cref{lem:char-sum-bound} to the final sum in \cref{eq:Ep-expansion} yields
    \[
        E_p(h_1,h_2,h_3,h_4) \ll p.
    \]
\end{itemize}
If $p$ is of Type I, II, or III, then \cref{eq:Ep-expansion} implies that $E_p(h_1,h_2,h_3,h_4) \asymp p^2$. We let $g, g', g''$ denote the products of all primes $p \mid \tfrac{c}{d}$ of Types I, II, and III respectively, and $G := gg'g''$; note that $g, g', g''$ depend on $(h_1,h_2,h_3,h_4)$. By plugging \cref{eq:E-product} and the corresponding bound for $E_p(h_1,h_2,h_3,h_4)$ into \cref{eq:C2-swap-sums}, then swapping sums and using the divisor bound, we obtain
\[
\begin{aligned}
    \msC_2 
    &\ll 
    \sum_{\substack{|h_1|,\ldots,|h_4| \le H \\ h_1 \equiv \cdots \equiv h_4 \pmod{d}}}
    \sum_{\substack{G = gg'g'' \mid \frac{c}{d} \\ g \mid (h_1-h_2, h_3-h_4) \\ g' \mid (h_1-h_3, h_2-h_4) \\ g'' \mid (h_1-h_4, h_2-h_3)}} c^{o(1)} G^2 \frac{c}{dG}
    \\
    &\ll c^{o(1)}
    \frac{c}{d}
    \max_{G = gg'g'' \mid \frac{c}{d}}
    G
    \sum_{\substack{|h_1|, \ldots, |h_4| \le H \\ h_1 \equiv \cdots \equiv h_4 \pmod{d} \\ h_1 \equiv h_2 \pmod{g},\ h_3 \equiv h_4 \pmod{g} \\ h_1 \equiv h_3 \pmod{g'},\ h_2 \equiv h_4 \pmod{g'} \\ h_1 \equiv h_4 \pmod{g''},\ h_2 \equiv h_3 \pmod{g''}}} 1.
\end{aligned}
\]
Finally, we pick $h_1, h_2, h_3, h_4$ in this order in the last sum, and use the assumption $d \le H$, to bound
\[
\begin{aligned}
    \msC_2
    &\ll 
    c^{o(1)} 
    \frac{c}{d}
    \max_{G = gg'g'' \mid \frac{c}{d}}
    G H \Big(1 + \frac{H}{dg}\Big) \Big(1 + \frac{H}{dg'g''}\Big)\Big(1 + \frac{H}{dgg'g''}\Big)
    \\
    &\ll 
    c^{o(1)} 
    \frac{cH}{d}
    \max_{G = gg'g'' \mid \frac{c}{d}}
    G \Big(\frac{H}{d} + \frac{H^2}{d^2 G}\Big)\Big(1 + \frac{H}{dG}\Big).  
\end{aligned}
\]
After expanding the product, each term is maximized either when $G = 1$ or when $G = \frac{c}{d}$. Using this and the assumption that $d \le H \ll c^{1+\eps}$, we find that
\[
\begin{aligned}
    \msC_2 &\ll 
    c^{o(1)} \frac{cH}{d} \Big(
    \Big(\frac{H}{d} + \frac{H^2}{d^2}\Big)\Big(1 + \frac{H}{d}\Big)
    +
    \frac{c}{d} \Big(\frac{H}{d} + \frac{H^2}{cd}\Big)\Big(1 + \frac{H}{c}\Big)
    \Big)
    \\
    &\ll_\eps 
    c^{O(\eps)}
    \frac{cH}{d} \Big(
    \frac{H^2}{d^2} \cdot \frac{H}{d}
    +
    \frac{c}{d} \cdot \frac{H}{d}\Big).
\end{aligned}
\]
This precisely recovers the desired bound.
\end{proof}

\section{Bounds for bilinear forms with Kloosterman sums}

In this section we collect various bounds for bilinear forms with Kloosterman sums whose performance in practice depends on the particular assumptions. At the end we combine them to a uniform formula. We start with a `trivial' bound generalizing \cref{eq:trivial-Weil,eq:trivial-Fourier}.

\begin{lemma} \label{lem:bilinear-forms-trivial-bound}
Let $c \in \Z_+$, $M, N \in \Z \cap [1, c]$, and $\mI, \mJ \subset \Z$ be intervals with $|\mI| = M$, $|\mJ| = N$. Then for any complex sequences $(\alpha_m)_{m \in \mI}$, $(\beta_n)_{n \in \mJ}$ and any $a \in (\Z/c\Z)^\times$, one has
\begin{equation} \label{eq:trivial-bound}
    \mathop{\sum\sum}_{\substack{m \in \mI, n \in \mJ \\ (m, n, c) = 1}} \alpha_m \beta_n S(am, n; c)
    \ll c^{o(1)} \|\alpha\| \|\beta\| \min(c, \sqrt{MNc}).
\end{equation}
If $\mI = \{1, \ldots, M\}$ and $\mJ = \{1, \ldots, N\}$, then \cref{eq:trivial-bound} also holds without the constraint $(m, n, c) = 1$.
\end{lemma}
\begin{proof}
The bound \cref{eq:trivial-bound} with the second term in the minimum follows immediately from the Weil bound \cref{eq:Weil-bound} and Cauchy--Schwarz. For \cref{eq:trivial-bound} with the first term in the minimum, see \cite[\S 3.2]{pascadi2025nonabelian}.

To prove the last statement in \cref{lem:bilinear-forms-trivial-bound}, we write
\begin{equation}\label{eq:remove-gcd-condition}
\begin{aligned}
    \sum_{m \le M} \sum_{n \le N} \alpha_m \beta_n S(am, n; c)
    &=
    \sum_{d \mid c}
    \mathop{\sum\sum}_{\substack{m' \le \frac{M}{d}, n' \le \frac{N}{d} \\ (m', n', \frac{c}{d}) = 1}} \alpha_{dm'} \beta_{dn'} \frac{\phi(c)}{\phi(c/d)} S\Big(am', n'; \frac{c}{d}\Big)
    \\
    &\ll 
    c^{o(1)}
    \sum_{d \mid c} d \|\alpha\| \|\beta\| \min\Bigg(\frac{c}{d}, \sqrt{\frac{MNc}{d^3}}\Bigg),
\end{aligned}
\end{equation}
and finish using the divisor bound.
\end{proof}

We now state the direct consequence of our approach, which works well when the square-full part $c_2$ of $c$ is small. 


\begin{theorem} \label{thm:bilinear-forms-4th-moment}
Let $c = c_1c_2$ where $c_1, c_2 \in \Z_+$, $(c_1, c_2) = 1$, $c_1$ is square-free, and $c_2$ is square-full. Let $M, N \in \Z \cap [1, c]$ and $\mI, \mJ \subset \Z$ be intervals with $|\mI| = M$, $|\mJ| = N$. Then for any complex sequences $(\alpha_m)_{m \in \mI}$, $(\beta_n)_{n \in \mJ}$ and any $a \in (\Z/c\Z)^\times$, one has
\begin{equation}\label{eq:bilinear-forms-4th-moment}
    \mathop{\sum\sum}_{\substack{m \in \mI, n \in \mJ \\ (m, n, c) = 1}} \alpha_m \beta_n S(am, n; c)
    \ll \|\alpha\| \|\beta\| c^{1+o(1)}  F(M,N,c,c_2)^{\frac{1}{4}},
\end{equation}
where
\begin{equation}\label{eq:F-bound}
\begin{split} 
    F(M,N,c,c_2) 
    &:= 
    \frac{c_2(M+N)MN}{c^2}  + F_0(M, N, c) ,
    \\
   F_0(M,N,c) &:=  \frac{M^{\frac{1}{2}} ((c+MN)(c+N^2))^{\frac{1}{4}}}{c} \min\Big(\frac{c}{M}, c^{\frac{1}{2}}\Big)^{\frac{1}{4}}
   +
    \Big(\frac{N^2}{c^2} + \frac{N^{\frac{1}{2}}M (c+N^2)}{c^{\frac{5}{2}}}\Big)^{\frac{1}{4}}   .
\end{split}
\end{equation}
If $\mI = \{1, \ldots, M\}$ and $\mJ = \{1, \ldots, N\}$, then \cref{eq:bilinear-forms-4th-moment} also holds without the constraint $(m, n, c) = 1$.
\end{theorem}

\begin{remark} 
When $M = N$, the saving factor in \cref{thm:bilinear-forms-4th-moment} simplifies to 
\[
    F(N,N,c,c_2)^{\frac{1}{4}}
    \asymp
    \frac{c_2^{1/4} N^{3/4}}{c^{1/2}} 
    +
    \frac{N^{1/8}}{c^{3/32}}
    +
    \frac{N^{5/16}}{c^{3/16}}.
\]
The last two of these terms also appear in \cref{thm:bilinear-forms-balanced-lengths}.
\end{remark}

\begin{proof}[Proof of \cref{thm:bilinear-forms-4th-moment}]

Let $\eps > 0$, $H_1 := 2c^{1+\eps} M^{-1}$, $H_2 := 2c^{1+\eps}N^{-1}$.
We use \cref{prop:kloosterman-to-characters,prop:holder-step} with $k = \ell = 2$ to obtain
\begin{equation}\label{eq:kloosterman-to-S0}
      \sum_{m \in \mI} \sum_{n \in \mJ} \alpha_m \beta_n S(am, n; c) \nu_{(m,n,c_1)} \one_{(m,n,c_2) = 1}  
    \ll_\eps
    \|\alpha\| \|\beta\|
    \frac{c^{1+3\eps} \msS_0^{\frac{1}{4}}}{\sqrt{H_1H_2}},
\end{equation}
where
\begin{equation}\label{eq:S0-bound}
    \msS_0 := (cc_2 + H_1H_2)(H_1 + H_2)
    +
    \sum_{\substack{d \in \Z_+ \\ c_2 \mid d \mid c,\, 2d \nmid c \\ d \le \min(H_1, H_2)}} 
    \sum_{\substack{\gamma \in \Z/d\Z \\ \gamma^2 = 1}} c_2 d \msA^{\frac{1}{2}} \msB^{\frac{1}{4}} \msC_2^{\frac{1}{4}},
\end{equation}
for $\msA$, $\msB$, $\msC_2$ as in \cref{eq:A-sum,eq:B-sum,eq:C-sum}. By combining \cref{prop:A-sum-bound,prop:B-sum-bound,prop:B-sum-bound-alt,prop:C-sum-bound}, we obtain
\begin{equation}\label{eq:combine-props}
\begin{aligned}
    &c_2d \msA^{\frac{1}{2}} \msB^{\frac{1}{4}} \msC_2^{\frac{1}{4}}
    \ll_\eps c^{O(\eps)} c_2d  
    \Big(\frac{H_1^2H_2}{d^2}\Big)^{\frac{1}{2}}
    \\ 
    &\times 
    \min \left\{\frac{H_1^2H_2}{d^2}\Big(1 + \frac{H_1H_2}{c}\Big)\Big(1 + \frac{H_1d}{c}\Big), \frac{H_1H_2}{d}\Big(\Big(   \frac{c^{1/2} }{d^{1/2}  }   
    + \frac{H_1^2}{c } \Big)\Big(1 + \frac{H_1H_2}{c}\Big) + \frac{H_1^2 H_2^{1/2}}{cd^{1/4}}\Big)\right\}^{\frac{1}{4}}
    \\
    &\times 
    \Big(\frac{c^2}{d} \Big(\frac{H_2}{d}\Big)^2 + c \Big(\frac{H_2}{d}\Big)^4\Big)^{\frac{1}{4}}.
\end{aligned}
\end{equation}
As functions of $d$, all terms in the expansion of this product grow at most like
$$
    d^{1-\frac{2}{2}} d^{-\frac{1}{4}} d^{-\frac{3}{4}}
    =
    d^{-1},$$
so as $d$ varies in $[c_2, c]$, the right-hand side of \cref{eq:combine-props} attains its maximum at $d = c_2$. Moreover, when $d = c_2$, the right-hand side of \cref{eq:combine-props} is non-increasing in $c_2$ (since all terms grow at most like $c_2 c_2^{-1}$), so we may replace all instances of $c_2, d$ with $1$ for an upper bound. 
This gives
\[
\begin{aligned}
    &c_2d \msA^{\frac{1}{2}} \msB^{\frac{1}{4}} \msC_2^{\frac{1}{4}}
    \ll_\eps c^{O(\eps)}  
     (H_1^2H_2 )^{1/2}(H_1H_2)^{1/4}
    \\ 
    &\times 
    \min\Bigg(H_1 \Big(1 + \frac{H_1H_2}{c}\Big),  \Big( c^{1/2}   +  \frac{H_1^2 }{c} \Big)\Big(1 + \frac{H_1H_2}{c}\Big) + \frac{H_1^2 H_2^{1/2}}{c}\Bigg)^{\frac{1}{4}} 
    \Big(c^2 H_2^2 + c H_2^4\Big)^{\frac{1}{4}}.
\end{aligned}
\]
Combining this with \cref{eq:S0-bound} (using the divisor bound and \cref{eq:center}) gives
\[
\begin{aligned}
    \msS_0 &\ll_\eps 
    (cc_2 + H_1H_2)(H_1+H_2)
    \\
    &+c^{O(\eps)}  
     (H_1H_2)^{5/4} 
     \Big[\min\Big(H_1, c^{1/2}   +  \frac{H_1^2 }{c}\Big)   (c + H_1H_2 ) + H_1^2 H_2^{1/2}\Big]^{1/4}  
     (c + H_2^2)^{1/4}.
\end{aligned}
\]
The term $H_1H_2(H_1+H_2)$ from the first line can   be ignored at this point, using $H_1 \ll c^{1+\eps}$. 
We then divide by $(H_1H_2)^2$ and recall that $H_1 = 2c^{1+\eps} M^{-1}$ and $H_2 = 2c^{1+\eps} N^{-1}$ to obtain
\begin{equation}\label{almostfinal}
\begin{split}
    \frac{\msS_0}{(H_1H_2)^2} 
    &\ll_\eps  c^{O(\varepsilon)} \tilde{F}(M, N, c, c_2)
    \end{split}
\end{equation}
where $ \tilde F(M, N, c, c_2)$ is defined as 
\begin{equation} \label{eq:F-tilde}
   \frac{c_2(M+N)MN}{c^2} + \frac{M^{\frac{1}{2}} ((c+MN)(c+N^2))^{\frac{1}{4}}}{c} \min\Big(\frac{c}{M}, c^{\frac{1}{2}} + \frac{c}{M^2}\Big) ^{\frac{1}{4}}  +\frac{N^{\frac{1}{8}}M^{\frac{1}{4}} (c+N^2)^{\frac{1}{4}}}{c^{5/8}}.
\end{equation}




Summarizing the above discussion (combining \cref{eq:kloosterman-to-S0,almostfinal}), we have proved
\begin{equation}\label{eq:kloosterman-bound-with-nu}
    \Big\vert \sum_{m \in \mI} \sum_{n \in \mJ} \alpha_m \beta_n S(am, n; c) \nu_{(m,n,c_1)} \one_{(m,n,c_2) = 1} \Big\vert
    \ll 
    \|\alpha\| \|\beta\| c^{1+o(1)} \tilde F(M, N, c, c_2)^{\frac{1}{4}}.
\end{equation}

Next, we replace the weight $\nu_{(m,n,c_1)}$ with $\one_{(m,n,c_1) = 1}$. For square-free $k, g \in \Z_+$, it follows quickly from \cref{eq:nu-function} and multiplicativity that
\[
    \one_{k = 1} = \sum_{g \mid k} f(g)\, \nu_{k/g}, 
    \qquad\quad 
    f(g) := \prod_{\text{prime }p \mid g} \frac{1}{p^2-1} \asymp g^{-2}.
\]
We apply this fact with $k = (m, n, c_1)$ to obtain
\begin{equation}\label{eq:pass-to-nu}
\begin{aligned}
    \mathop{\sum\sum}_{\substack{m \in \mI, n \in \mJ \\ (m, n, c) = 1}} \alpha_m \beta_n S(am, n; c)
    &=
    \sum_{m \in \mI} \sum_{n \in \mJ} \sum_{g \mid (m, n, c_1)} f(g)\, \nu_{(\frac{m}{g},\frac{n}{g},\frac{c_1}{g})}  \alpha_m \beta_n S(am, n; c)  \one_{(m,n,c_2) = 1}
    \\
    &= 
    \sum_{g \mid c_1} f(g) \frac{\phi(c)}{\phi(c/g)} \sum_{\substack{m \in \mI \\ g \mid m}} \sum_{\substack{n \in \mJ \\ g \mid n}} \alpha_m \beta_n S(a\tfrac{m}{g}, \tfrac{n}{g}; \tfrac{c}{g}) \nu_{(\frac{m}{g},\frac{n}{g},\frac{c_1}{g})} \one_{(\frac{m}{g},\frac{n}{g},c_2) = 1}.
\end{aligned}
\end{equation}
Let $\mI_g := \{m' \in \Z : gm' \in \mI\}$ and $\mJ_g := \{n' \in \Z : gn' \in \mJ\}$, so $|\mI_g| \ll 1 + \tfrac{M}{g}$ and $|\mJ_g| \ll 1 + \tfrac{N}{g}$. If one of these intervals is empty, then the last sum over $m, n$ above vanishes. Otherwise, we can apply \cref{eq:kloosterman-bound-with-nu} with the choice of parameters
\[
    (c, c_1, c_2, \mI, \mJ, (\alpha_m)_{m \in \mI}, (\beta_n)_{n \in \mJ}) \gets (\tfrac{c}{g}, \tfrac{c_1}{g}, c_2, \mI_g, \mJ_g, (\alpha_{gm'})_{m' \in \mI_g}, (\beta_{gn'})_{n' \in \mJ_g})
\] 
to obtain
\[
    \sum_{\substack{m \in \mI \\ g \mid m}} \sum_{\substack{n \in \mJ \\ g \mid n}} \alpha_m \beta_n S(a\tfrac{m}{g}, \tfrac{n}{g}; \tfrac{c}{g}) \nu_{(\frac{m}{g},\frac{n}{g},\frac{c_1}{g})} \one_{(\frac{m}{g},\frac{n}{g},c_2) = 1}
    \ll 
    \|\alpha\| \|\beta\| \Big(\frac{c}{g}\Big)^{1+o(1)} \tilde F(1 + \tfrac{M}{g}, 1 + \tfrac{N}{g}, \tfrac{c}{g}, c_2)^{\frac{1}{4}}.
\]
From this, \cref{eq:pass-to-nu}, and the bounds $f(g) \asymp g^{-2}$, $\phi(c) \le g\phi(c/g)$, we deduce that
\[
    \mathop{\sum\sum}_{\substack{m \in \mI, n \in \mJ \\ (m, n, c) = 1}} \alpha_m \beta_n S(am, n; c)
    \ll 
    \|\alpha\|\|\beta\| c^{1+o(1)}
    \max_{g \mid c_1} g^{-2}
    \tilde F(1 + \tfrac{M}{g}, 1 + \tfrac{N}{g}, \tfrac{c}{g}, c_2)^{\frac{1}{4}}
\]
Upon inspecting \cref{eq:F-tilde}, we see that the maximum above is attained when $g = 1$, so \cref{eq:bilinear-forms-4th-moment} holds with $F$ replaced by $\tilde F$.

Finally, let us simplify the factor $\tilde F(M, N, c, c_2)$. We may take the term $ c/M^2$ out of the minimum in \cref{eq:F-tilde}; this gives a contribution of $(c+MN)(c+N^2))^{1/4}/c^{3/4}$, which is only relevant if $M \le c^{1/4}$. After expanding the product $(c+MN)(c+N^2)$, the contribution of the $c \cdot c$ term is $c^{-1/4}$, and this is only relevant when $N \le c^{1/2}$. This gives
\[
\begin{aligned}
    \tilde F(M, N, c, c_2) &\asymp \frac{c_2(M+N)MN}{c^2}  + \frac{M^{1/2} ((c+MN)(c+N^2))^{1/4}}{c} \min\Big(\frac{c}{M}, c^{1/2}\Big)^{1/4}
    \\
    &+
    \Big(\frac{cN^2 + cMN + N^3M}{c^3} + \frac{N^{1/2}M  (c+N^2) }{c^{5/2}}\Big)^{1/4} + \one_{M \le c^{1/4}} \one_{N \le c^{1/2}}\ c^{-1/4}.
\end{aligned}
\]
In the third term above, we have $cMN/c^3 \le N^{1/2}Mc/c^{5/2}$ and $N^3M/c^3 \le N^{1/2}MN^2/c^{5/2}$, so a couple of inner terms can be ignored. Moreover, when $M \le c^{1/4}$ and $N \le c^{1/2}$, the second term above is larger than $ M^2N^2/c^2$, which gives an acceptable bound on its own by \cref{lem:bilinear-forms-trivial-bound}. Thus the last term $\one_{M \le c^{1/4}} \one_{N \le c^{1/2}}\ c^{-1/4}$ can also be ignored, which proves \cref{eq:bilinear-forms-4th-moment}.


When $\mI = \{1, \ldots, M\}$ and $\mJ = \{1, \ldots, N\}$, one can remove the constraint $(m, n, c) = 1$ by the same argument as in \cref{eq:remove-gcd-condition}. Indeed,
writing $c_{2,d}$ for the square-full part of $c/d$, we have $F(\tfrac{M}{d},\tfrac{N}{d},\tfrac{c}{d},c_{2,d}) \le F(\tfrac{M}{d},\tfrac{N}{d},\tfrac{c}{d},c_2)$, and the expression $\frac{c}{d} \cdot dF(\tfrac{M}{d},\tfrac{N}{d},\tfrac{c}{d},c_2)^{1/4}$ is non-increasing in $d$. 
\end{proof}

\begin{theorem}[\cite{pascadi2025nonabelian}] \label{thm:non-ab-result}
Let $c = dd'e$ where $d, d', e \in \Z_+$, $d' \mid d$, and $(d, e) = 1$. Let $M, N \in \Z \cap [1, c]$ and $\mI, \mJ \subset \Z$ be intervals with $|\mI| = M$, $|\mJ| = N$. Then for any complex sequences $(\alpha_m)_{m \in \mI}$, $(\beta_n)_{n \in \mJ}$ and any $a \in (\Z/c\Z)^\times$, one has
\[
    \mathop{\sum\sum}_{\substack{m \in \mI, n \in \mJ \\ (m, n, c) = 1}} \alpha_m \beta_n S(am, n; c) 
    \ll
    \|\alpha\| \|\beta\| c^{1+o(1)}
    G(M,N,c,d)^{\frac{1}{6}},
\]
where, writing $f$ for the maximal positive integer satisfying $f^2 \mid cd$, we have
\[
    G(M,N,c,d) := \frac{dMN (M^2+N^2)}{c^3} + \frac{f(M^2+N^2)}{c^2} + \frac{f}{d^2}.
\]
\end{theorem}

\begin{proof}
This is \cite[Theorem 7.1]{pascadi2025nonabelian}, with the assumption that $M \ge N$ removed by symmetrizing the upper bound.
\end{proof}

By combining \cref{thm:bilinear-forms-4th-moment,thm:non-ab-result}, we deduce a result that does not depend on the factorization of $c$. For nearly-square-free moduli, it is better to use \cref{thm:bilinear-forms-4th-moment} directly.

\begin{theorem} \label{thm:bilinear-forms-general}
Let $c \in \Z_+$, $M, N \in \Z \cap [1, c]$, and $\mI, \mJ \subset \Z$ be intervals with $|\mI| = M$, $|\mJ| = N$. Then for any complex sequences $(\alpha_m)_{m \in \mI}$, $(\beta_n)_{n \in \mJ}$ and any $a \in (\Z/c\Z)^\times$, one has
\begin{equation}\label{eq:bilinear-forms-general}
\begin{aligned}
    \mathop{\sum\sum}_{\substack{m \in \mI, n \in \mJ \\ (m, n, c) = 1}} \alpha_m \beta_n S(am, n; c) 
    \ll
    \|\alpha\| \|\beta\| c^{1+o(1)}
    H(M,N,c),
\end{aligned}
\end{equation}
where
\[
\begin{aligned}
    H(M, N, c) &:= 
      \frac{M^{1/8} ((c+MN)(c+N^2))^{1/16}}{c^{1/4}} \min\Big(\frac{c}{M}, c^{1/2}\Big)^{1/16}
      + \Big(\frac{N^2}{c^2} + \frac{N^{1/2}M  (c+N^2) }{c^{5/2}}\Big)^{1/16}
    \\
    &+
    \frac{M^{1/3} + N^{1/3}}{c^{1/5}}
    +
    \frac{M^{1/2}N^{1/6} + M^{1/6}N^{1/2}}{c^{7/18}}
    +
    \frac{M^{1/15} + N^{1/15}}{c^{1/15}}.
\end{aligned}
\]
If $\mI = \{1, \ldots, M\}$ and $\mJ = \{1, \ldots, N\}$, then \cref{eq:bilinear-forms-general} also holds without the constraint $(m, n, c) = 1$. 
\end{theorem}

\begin{proof}[Proof of \cref{thm:bilinear-forms-general}]
Let $c_1, c_2 \in \Z_+$ be such that $c_1c_2 = c$, $c_1$ is square-free, $c_2$ is square-full, and $(c_1, c_2) = 1$. We start by obtaining two bounds from \cref{thm:non-ab-result}. First, we apply it with $d = f = c_2$, $d' = 1$, $e = c_1$, in which case
\begin{equation}\label{g1}
G(M, N, c, d) = c_2\frac{M^2 + N^2}{c^2} \Big( 1 + \frac{MN}{c}\Big) + \frac{1}{c_2}.
\end{equation}
Next we apply it in a more sophisticated way, where we roughly choose $d \approx c_2^{1/2}$. More precisely, let us write
\[
    c_2 = \prod_{j=2}^\infty s_j^j = s_2^2 s_3^3 s_4^4 \cdots,
\]
where the $s_j$ are pairwise coprime and squarefree (and only finitely many are different from 1). With this notation, we choose 
\[
    d = \prod_{j=2}^{\infty} s_j^{\lceil j/2 \rceil} = s_2 s_3^2 s_4^2 \cdots,
    \qquad 
    d' = \prod_{j=2}^{\infty} s_j^{\lfloor j/2 \rfloor} = s_2 s_3 s_4^2 \cdots, 
    \qquad
    e = c_1,
\]
so that
$$ f = \prod_{j=2}^{\infty} s_j^{\lf \frac{j + \lceil j/2 \rceil}{2} \rf}.$$
It is easy to see that
$$f \leq c_2^{4/5}, \quad d \leq c_2^{2/3}, \quad \frac{f}{d^2} \leq c_2^{-1/4},$$
so that we obtain
\begin{equation}\label{g2}
G(M, N, c, d) \le  \frac{M^2 + N^2}{c^2} \Big( c_2^{4/5} + c_2^{2/3}\frac{MN}{c}\Big) + \frac{1}{c_2^{1/4}}.
\end{equation}
Combining \cref{thm:bilinear-forms-4th-moment} and \cref{thm:non-ab-result} with the bounds \eqref{g1} and \eqref{g2}, we obtain
\begin{equation}\label{eq:bil-kl-bound-by-H}
    \begin{split}
    \mathop{\sum\sum}_{\substack{m \in \mI, n \in \mJ \\ (m, n, c) = 1}} \alpha_m \beta_n S(am, n; c) 
    \ll
    \|\alpha\| \|\beta\| c^{1+o(1)} H(M,N,c,c_2)
\end{split}
\end{equation}
(and without the condition $(m, n, c) = 1$ if $\mI = \{1, \ldots, M\}$, $\mJ = \{1, \ldots, N\}$), where
\begin{displaymath}
    \begin{split}
    &H(M,N,c,c_2) := F_0(M, N, c) 
    +\Big( \frac{M^2 + N^2}{c^2} \Big( c^{4/5} + c^{2/3}\frac{MN}{c}\Big) \Big)^{1/6} \\
& + \min\Bigg(\Big(c_2\frac{(M+N)MN}{c^2}\Big)^{1/4} , \Big(c_2\frac{M^2 + N^2}{c^2} \Big( 1 + \frac{MN}{c}\Big) + \frac{1}{c_2}\Big)^{1/6}, \Big(\frac{1}{c_2^{1/4}}\Big)^{1/6} \Bigg).
\end{split}
\end{displaymath}
In the second line, the minimum of the first two entries is bounded by
$$ 
\Big(\frac{(M+N)MN}{c^2}\Big)^{\frac{1}{4} \cdot \frac{2}{5}} + \Big(c_2\frac{M^2 + N^2}{c^2} \Big( 1 + \frac{MN}{c}\Big)\Big)^{1/6} .   
$$
Combining this with the last entry, we obtain 
$$ \Big(\frac{(M+N)MN}{c^2}\Big)^{1/10} + \Big( \frac{M^2 + N^2}{c^2} \Big( 1 + \frac{MN}{c}\Big)\Big)^{\frac{1}{6} \cdot \frac{1}{5}}.  $$
Thus altogether we have 
\begin{equation}\label{eq:H-combined-bound}
    \begin{split}
   H(M,N,c,c_2)  &\ll  F_0(M, N, c) + 
    \Big( \frac{M^2 + N^2}{c^{6/5}}  + \frac{(M^2 + N^2)MN}{c^{7/3}} \Big)^{1/6} \\
    &
    +  \Big(\frac{(M+N)MN}{c^2}\Big)^{1/10} + \Big( \frac{M^2 + N^2}{c^2} \Big( 1 + \frac{MN}{c}\Big)\Big)^{1/30}.
    \end{split}
\end{equation}
The first term in the second line of \cref{eq:H-combined-bound} is dominated by  $\tfrac{M^{1/3} + N^{1/3}}{c^{1/5}}$ from the first line, so it can be ignored. The remaining portion of the second line is
\[
    \asymp \frac{M^{1/15} + N^{1/15}}{c^{1/15}} + \frac{M^{1/10}N^{1/30} + M^{1/30}N^{1/10}}{c^{1/10}},
\]
and the second term here can also be ignored, since it is dominated by $\tfrac{M^{1/3} + N^{1/3}}{c^{1/5}} + \tfrac{M^{1/15} + N^{1/15}}{c^{1/15}}$. 
Overall, we obtain
\[
\begin{aligned}
    H(M,N,c,c_2)
    &\ll F_0(M, N, c)  + 
    \frac{M^{1/3} + N^{1/3}}{c^{1/5}}
    +
    \frac{M^{1/2}N^{1/6} + M^{1/6}N^{1/2}}{c^{7/18}}
    +
    \frac{M^{1/15} + N^{1/15}}{c^{1/15}}.
\end{aligned}
\]
Plugging this into \cref{eq:bil-kl-bound-by-H} gives the desired result.
\end{proof}

\begin{proof}[Proof of \cref{thm:bilinear-forms-balanced-lengths}]
We may assume without loss of generality that $|\mI| = |\mJ| = N$, by extending the sequences $(\alpha_m)$, $(\beta_n)$ with zeros. 
Taking $M = N$ in \cref{thm:bilinear-forms-general} gives a saving factor of
\[
    H(N, N, c) \asymp  \frac{N^{1/8}(c + N^2)^{1/8}}{c^{1/4}} \min\Big( \frac{c}{N}, c^{1/2}\Big)^{1/16} + \Big(\frac{N^{3/2}}{c^{3/2}} + \frac{N^{7/2}}{c^{5/2}}\Big)^{1/16}
    +
    \frac{N^{1/3}}{c^{1/5}}
    +
    \frac{N^{2/3}}{c^{7/18}}
    +
    \frac{N^{1/15}}{c^{1/15}}.
\]
It is straightforward to check that this equals
$$\begin{cases} 
N^{2/3}/c^{7/18}, & N \geq c^{29/51},\\ 
N^{5/16}{c^{3/16}}, & c^{1/2} \leq N \leq c^{29/51},\\ 
N^{1/8}/c^{3/32}, & c^{13/28} \leq N \leq c^{1/2},\\ 
N^{1/15}/c^{1/15}, & N \leq c^{13/28},
\end{cases}$$
so that 
$$H(N, N, c) \ll \frac{N^{1/15}}{c^{1/15}} + \frac{N^{1/8}}{c^{3/32}} + \frac{N^{5/16}}{c^{3/16}} + \frac{N^{2/3}}{c^{7/18}}. $$

The first term can be dropped using the trivial bound in \cref{lem:bilinear-forms-trivial-bound}. Indeed, if $N \le c^{13/28}$, then $c \cdot \tfrac{N^{1/8}}{c^{3/32}} \ge N\sqrt{c}$, so the term $N^{1/8}/c^{3/32}$ is enough to give a correct upper bound on its own. On the other hand, if $N > c^{13/28}$, the first term is dominated by $N^{1/8}/c^{3/32}$.
This completes the proof.
\end{proof}



We complement \cref{thm:bilinear-forms-general} with one more bound, which is useful for unbalanced intervals. We start with a preparatory lemma
 (see \cite[Proposition 23]{blomer2015second} for a version with odd $c$). For $m, n, h \in \mathbb{Z}$ and $c \in \mathbb{Z}_+$ consider the character sum
$$S(m ,n, h; c) = \sum_{\substack{x \, (\text{mod } c)\\ (x(x+h) , c) = 1}} e\Big(\frac{m\bar{x} + n\overline{(x+h)}}{c}\Big).$$
We recall the notation \eqref{sqrtgcd}. 

\begin{lemma}\label{lem-expsum}
We have $S(m ,n, h; c) \ll c^{1/2+o(1)} (m+n, hm, hn, c^{1/2})$.
\end{lemma}

As an immediate corollary we see that
\begin{equation}\label{coro}
\sum_{\nu \, (\text{mod } c)}  S(m, \nu, c)S(n, \nu, c)\, e\Big(\frac{\nu h}{c}\Big) \ll c^{3/2+o(1)} (m-n, hm, hn, c^{1/2})
\end{equation}
which follows simply by opening the two Kloosterman sums. 

Before we start with the proof we recall a bound for exponential sums in one variable: if $R = P/Q \in \mathbb{F}_p(x)$ is a rational function not of the form $R = h^p - h$ for some $h \in \overline{\mathbb{F}}_p(x)$ and $\deg P \leq \deg Q$, then
\begin{equation}\label{expsum}
\Big|\sum_{\substack{x\, (\text{mod } p)\\ Q(x) \not = 0}} e\Big(\frac{R(x)}{p}\Big)\Big| \leq (2\deg Q -2)\sqrt{p} + 1;
\end{equation}
see e.g.\ \cite[Theorem 2]{moreno1991exponential}. 

\begin{proof} Let $\delta = (m, n, c)$ with $m = \delta m'$, $n = \delta n'$, $c = \delta c'$, then we have
\begin{equation}\label{delta}
|S(m, n, h, c)| \leq  \delta|S (m', n', h, c')|.
\end{equation}
 It   suffices to bound $S (m', n', h, c')$ for $c' = p^k$ a prime power, and without loss of generality we may then assume $p \nmid m'$. 

Suppose that $k = 1$. If $p \mid n'$, then $S(m', n', h; p)  = r_{m'}(p) - e(-m'\bar{h}/p) = O(1)$. If $p \mid h$, then $|S(m', n', h; p)| = |r_{m'+n'}(p)| \leq (m'+n', p).$ Finally, if $p \nmid hn$, then with the rational function $m/x + n/(x+h) \in \mathbb{F}_p(x)$ we have   $|S(m', n', h; p)| \leq 2\sqrt{p} + 1$ by  \eqref{expsum}. 
Thus in total we obtain 
\begin{equation}\label{k1}
   S(m', n', h, p) \ll \sqrt{p} (m'+n', h, p)^{1/2}.
\end{equation}

Now suppose that $k  = 2\ell \geq 2$ is even. Then by a standard stationary phase argument \cite[Lemma 12.2]{iwaniec2004analytic}, we have
$$|S(m', n', h; p^{k}) | \leq p^{\ell} \#\{ x\, (\text{mod } p^{\ell}) \mid m'\bar{x}^2 + n'\overline{(x+h)^2} \equiv 0 \, (\text{mod } p^{\ell})\}.$$ 
Clearly we must have $(n', p) = 1$, and we obtain a quadratic congruence
$$x^2 (m' + n') + 2 x h m' + h^2m' \equiv 0 \, (\text{mod } p^{\ell})$$
where $p \nmid x(x+h)$. Let $d = (m' + n', 2h, p^{\ell})$, then the congruence is equivalent to
\begin{equation}\label{cong}
x^2 \frac{m' + n'}{d} +   x \frac{2h}{d} m' + \frac{h^2}{d}m' \equiv 0 \, \Big(\text{mod } \frac{p^{\ell}}{d}\Big)
\end{equation}
where at least one of $(m' + n')/d$ and $h/d$ is coprime to $p$ and necessarily $d 
\mid h^2$. 
Without loss of generality assume that $d < p^{\ell}$, otherwise there is nothing to do.

Suppose first that $(m' + n')/d$ is coprime to $p$. If $p=2$, assume in addition that $2h/d$ is even, so $h/d$ is integral. 
 Then completing the square we get
$$\Big(x + \frac{h}{d} m' \overline{\frac{m'+n'}{d}}\Big)^2 + \Big(\frac{h}{d}\Big)^2 m'n'  \Big(\overline{\frac{m'+n'}{d}}\Big)^2 \equiv 0 \, \Big(\text{mod } \frac{p^{\ell}}{d}\Big).$$ 
Since $p \nmid x$, this is only possible if $p \nmid h/d$, and then by Hensel's lemma, there are at most 2 solutions if $p$ is odd and at most 4 solutions if $p= 2$. 

Suppose next that $2h/d$ is coprime to $p$, but $(m' + n')/d$ is not, and suppose that $x, x'$ are two solutions to \eqref{cong}. Then we obtain
$$(x  - x') \Big( \frac{m' + n'}{d} (x + x') + \frac{2h}{d} m' \Big) \equiv 0 \, \Big(\text{mod }  \frac{p^{\ell}}{d}\Big).$$
The second factor is a unit modulo $p$, hence $x \equiv x'$ (mod $p^{\ell}/d$), and there is at most one solution.

Finally, if $p = 2$, there is the case when both $2h/d$ and $(m'+n')/d$ are odd. In this case, the derivative of the quadratic function on the left hand side of \eqref{cong} is always odd, hence by Hensel's lemma every solution modulo 2 (of which there are at most 2) lifts uniquely.

 In total we deduce that \eqref{cong} has at most $4$ solutions with $p \nmid x$, and hence
\begin{equation}\label{keven}
|S(m', n', h; p^{k}) |\leq 4p^{\ell} d = 4p^{\ell}(m' + n', 2h, p^{\ell}).
\end{equation}

Finally suppose that $k  = 2\ell + 1 \geq 3$ is odd. Then again by a stationary phase argument \cite[Lemma 12.3]{iwaniec2004analytic} we have
\begin{equation}\label{koddprelim}
|S(m', n', h; p^{k}) | \leq \sum_{x\, (\text{mod } p^{\ell})} \Big|\sum_{z \, (\text{mod } p^{\ell+1})} e\Big(-z \frac{ m'\bar{x}^2 + n'\overline{(x+h)^2} }{p^{\ell+1}} +  z^2 \frac{ (m'\bar{x}^3 + n'\overline{(x+h)^3} )} {p}\Big)\Big|.
\end{equation}
This vanishes unless $m'\bar{x}^2 + n'\overline{(x+h)^2} \equiv 0 \, (\text{mod } p^{\ell})$. Hence again we must have $p \nmid n'$, and if $p \mid h$, then necessarily $p\mid m' + n'$. 

Let us first assume $p \nmid h$. Then $|S(m', n', h; p^{k}) |$ is bounded by
$$   p^{\ell} \sum_{\substack{x\, (\text{mod } p^{\ell})\\ m'\bar{x}^2 + n'\overline{(x+h)^2} \equiv 0 \, (\text{mod } p^{\ell})}} \Big|\sum_{z \, (\text{mod } p )} e\Big(-z \frac{ (m'\bar{x}^2 + n'\overline{(x+h)^2} )/p^{\ell}}{p} +  z^2 \frac{ (m'\bar{x}^3 + n'\overline{(x+h)^3} )} {p}\Big)\Big|.$$
Under our current assumption we have 
$$m'\bar{x}^3 + n'\overline{(x+h)^3} \equiv (m'\bar{x}^2 + n'\overline{(x+h)^2} )\bar{x} - n' h \overline{x(x+h)^3} \equiv  - n' h \overline{x(x+h)^3}  \not\equiv 0 \, (\text{mod } p),$$
so that the $z$-sum is $O(p^{1/2})$ and the complete contribution is $O(p^{\ell + 1/2})$. 

Let us now assume that $p \mid (m'+n', h)$. Then the $z^2$-term in \eqref{koddprelim} disappears, and we enlarge the $x$-sum artificially to a sum modulo $p^{\ell+1}$. This gives us
\begin{displaymath}
    \begin{split}
&\frac{1}{p} \sum_{x \, (\text{mod } p^{\ell+1})} \Big|\sum_{z \, (\text{mod } p^{\ell+1})} e\Big(-z \frac{ m'\bar{x}^2 + n'\overline{(x+h)^2} }{p^{\ell+1}} \Big)\Big| =  p^{\ell} \sum_{\substack{x\, (\text{mod } p^{\ell+1})\\ m'\bar{x}^2 + n'\overline{(x+h)^2} \equiv 0 \, (\text{mod } p^{\ell+1})}}1\\
&\ll p^{\ell}(m' + n', h, p^{\ell+1}).
\end{split}
\end{displaymath}


Combining this analysis with \eqref{k1} and \eqref{keven}, 
we conclude that  in all cases we obtain
$$
S(m', n', h; p^{k})  \ll p^{k/2} (m' + n', h, p^{k/2}).
$$
Substituting back into \eqref{delta}, we obtain
$$S(m, n, h, p^r) \ll  \delta  \Big(\frac{p^r}{\delta} \Big)^{1/2} \Big(\frac{m + n}{\delta}, h,  \Big(\frac{p^r}{\delta} \Big) \Big)^{1/2} \ll p^{r/2} (m + n, hm, hn, p^{r/2}),$$
which generalizes   to all  moduli by the Chinese remainder theorem. 
\end{proof}
 
The previous lemma is the key input for the  following generalization of \cite[Theorem 1.17]{fouvry2014algebraic} to arbitrary moduli, which may be of independent interest.

\begin{theorem}\label{polyaV}
    Let $c  \in \Z_+$, $M, N \in \Z $, and $\mI, \mJ \subset \Z$ be intervals with $|\mI| = M$, $|\mJ| = N$. Then for any complex sequences $(\alpha_m)_{m \in \mI}$, $(\beta_n)_{n \in \mJ}$ and any $a \in (\Z/c\Z)^\times$, one has
\begin{equation*} \label{eq:polyaV}
    \mathop{\sum\sum}_{\substack{m \in \mI, n \in \mJ \\ (m, c) = 1
    }} \alpha_m \beta_n S(am, n; c)
    \ll \|\alpha\| \|\beta\| c^{1+o(1)}\Big( \frac{(MN)^{1/2}}{c^{3/4}} + \frac{N^{1/2}}{c^{1/2}} + \frac{M^{1/2}}{c^{1/4}}\Big). 
\end{equation*}
\end{theorem}

\begin{proof}
By Cauchy's inequality, the sum in question is bounded by 
$$\| \beta \| \Big(\sum_{\substack{m_1, m_2 \in \mI \\ (m_1m_2, c) = 1}} \alpha_{m_1} \overline{\alpha_{m_2}} \sum_{n \in \mJ} S(am_1, n, c) S(am_2, n, c)\Big)^{1/2}.$$
By completing the $n$-sum (see \cite[(12.11) -- (12.13)]{iwaniec2004analytic}), this is bounded by
$$\| \beta \| \Big(\sum_{\substack{m_1, m_2 \in \mI\\ (m_1m_2, c) = 1}} |\alpha_{m_1} \overline{\alpha_{m_2}}|  \sum_{0 \leq |h| \leq c/2}  \min\Big(\frac{1}{|h|}, \frac{N}{c}\Big)|S_{m_1, m_2}(h)|  \Big)^{1/2}$$
with $$ S_{m_1, m_2}(h) =  \sum_{\nu \, (\text{mod } c)}  S(m_1, \nu, c)S(m_2, \nu, c) e\Big(\frac{\nu h}{c}\Big).$$
Using \eqref{coro}, we obtain
\begin{displaymath}
\begin{split}
&\| \beta \| \Big(\sum_{\substack{m_1, m_2 \in \mI\\ (m_1m_2, c) = 1 }} |\alpha_{m_1} \overline{\alpha_{m_2}}|  \sum_{0 \leq |h| \leq c/2}  \min\Big(\frac{1}{|h|}, \frac{N}{c}\Big)c^{3/2+ o(1)} (m_1 - m_2, h, 
c^{1/2})  \Big)^{1/2}\\
\ll &\| \beta \| \Big(\sum_{\substack{d \mid c }} (d, c^{1/2})  \sum_{\substack{m_1, m_2 \in \mI\\ d \mid m_1 - m_2}} |\alpha_{m_1} \overline{\alpha_{m_2}}|  \sum_{ 0 \leq |h| \ll c
/d  } \min\Big(\frac{1}{d|h| 
}, \frac{N}{c}\Big)c^{3/2+ o(1)}   \Big)^{1/2}\\
\ll &\| \beta \| \Big(\sum_{\substack{d \mid c }} (d, c^{1/2})  \sum_{\substack{m_1, m_2 \in \mI\\ d \mid m_1 - m_2}} |\alpha_{m_1} \overline{\alpha_{m_2}}|   Nc^{1/2 + o(1)} \Big(1 + \frac{c
}{Nd}  \Big)\Big)^{1/2}\\
\ll &\| \beta \| \Big(\sum_{\substack{d \mid c  }} (d, c^{1/2})   
\| \alpha \|^2 \Big(1 + \frac{M}{d
}\Big)   Nc^{1/2 + o(1)} \Big(1 + \frac{c
}{Nd}  \Big)\Big)^{1/2}\\
\end{split}
\end{displaymath}
and the claim follows.
\end{proof}

\begin{remark}\label{largesize}
   \cref{thm:bilinear-forms-4th-moment,thm:non-ab-result,thm:bilinear-forms-general} as well as \cref{thm:bilinear-forms-balanced-lengths} contain the condition $N, M \leq c$. If $N$ is larger than $c$, we observe that the Kloosterman sum depends only on $n$ modulo $c$, hence we can split the interval $\mJ$ into $(1 + N/c)$ subintervals $\mJ_i$ of length at most $c$ and use the Cauchy-Schwarz inequality to deduce $$\Big\| \sum_i \beta|_{\mJ_i} \Big\| \leq \Big(1 + \frac{N}{c}\Big)^{1/2} \| \beta \|. $$  
   In other words, we can drop the condition $N \leq c$ at the cost of multiplying all bounds by $(1 + N/c)^{1/2}$. The same reasoning applies for the $m$-variable.
\end{remark}

\section{Moments of twisted \texorpdfstring{$L$}{L}-functions}\label{sec6}


In this section we prove \cref{secondmoment}. The proof follows well-known steps, and we refer to \cite{blomer2015second, blomer2017moments, milicevic2025bilinear} for details. Our task here is mainly to collect the various auxiliary bounds and optimize them. 

The key input is to bound 
$$\mathcal{B}^{\pm}(M, N) = \frac{1}{\phi^{\ast}(q)} \sum_{d\mid q} \mu\Big(\frac{q}{d}\Big) \frac{\phi(d)}{(MN)^{1/2}} \sum_{\substack{m \equiv \pm n\, (\text{mod } d)\\ (mn, q) = 1\\ m \not = n}} \lambda_1(m) \lambda_2(n) W_1\Big(\frac{m}{M}\Big) W_2\Big(\frac{n}{N}\Big)$$
for two smooth and compactly supported functions $W_1, W_2$ and 
\begin{equation}\label{ranges}
N \geq M, \quad MN \leq q^{2+\varepsilon}.
\end{equation}
As usual, we  write $\theta = 7/64$ for an admissible exponent towards the Ramanujan-Petersson conjecture. We have the trivial bound
\begin{equation}\label{bound1}
    \mathcal{B}^{\pm}(M, N) \ll q^{o(1)}N^{\theta} \frac{(MN)^{1/2}}{q}.
\end{equation}
We also have the bound from shifted convolution sums (\cite[Lemma 7.1]{milicevic2025bilinear}, \cite[Theorem 3.2]{blomer2015second}):
\begin{equation}\label{bound2}
    \mathcal{B}^{\pm}(M, N) \ll q^{o(1)}\Big(\Big( \frac{N}{M}\Big)^{1/4} q^{-1/4} + \Big( \frac{N}{M}\Big)^{1/2} q^{-1/2}+ q^{-1/2 + 2\theta}\Big).
\end{equation}
Finally, by the discussion in \cite[(7.18) -- (7.19)]{milicevic2025bilinear} we have
\[
\mathcal{B}^{\pm}(M, N) \ll q^{o(1)}\sum_{r \mid q} \sum_{f \mid g \mid \frac{q}{(q, r^{\infty})}} \frac{f^{\theta}}{q (fgMN^{\ast})^{1/2}} \Big|\sum_{\substack{m \asymp M\\ (m, q) = 1 }} \sum_{n \asymp N^{\ast}} \alpha_m \beta_n S(\pm \overline{fg}n, m, r)\Big| 
\]
where $\|\alpha \| \ll M^{1/2}$, $\| \beta \| \ll (N^*)^{1/2}$ and $N^{\ast} \ll fgr^2/N$.  

For the inner sum, we will use both \cref{thm:bilinear-forms-general} and \cref{polyaV}, the former together with Remark \ref{largesize}. 
In this way we obtain 
\begin{displaymath}
    \begin{split}
\mathcal{B}^{\pm}&(M, N) \ll q^{o(1)}\sum_{r \mid q} \sum_{f \mid g \mid \frac{q}{r}} \frac{f^{\theta}r}{q (fg)^{1/2}}  \min\Bigg(  \frac{(MN^{\ast})^{1/2}}{r^{3/4}} + \frac{(N^{\ast})^{1/2}}{r^{1/2}} + \frac{M^{1/2}}{r^{1/4}}, \\&\Big(1 + \frac{M^{1/2}}{r^{1/2}}\Big) \Big(1 + \frac{(N^{\ast})^{1/2}}{r^{1/2}}\Big)\Big[\frac{M^{1/8} ((r+MN^{\ast})(r+(N^{\ast})^2))^{1/16}}{r^{1/4}} \min\Big(\frac{r}{M}, r^{1/2}\Big)^{1/16} 
    \\
    &\quad\quad\quad\quad +\Big(\frac{(N^{\ast})^2}{r^2} + \frac{(N^{\ast})^{1/2}M  (r+(N^{\ast})^2) }{r^{5/2}}\Big)^{1/16}+
    \frac{M^{1/3} + (N^{\ast})^{1/3}}{r^{1/5}}
   \\& \quad\quad\quad\quad +
    \frac{M^{1/2}(N^{\ast})^{1/6} + M^{1/6}(N^{\ast})^{1/2}}{r^{7/18}}
    +
    \frac{M^{1/15} + (N^{\ast})^{1/15}}{r^{1/15}} \Big]\Bigg).
\end{split}
\end{displaymath} 
The expression is obviously increasing in $N^{\ast}$, so we may substitute $N^{\ast} = fgr^2/N \leq q^2/N$. We substitute $q^2/N$ in all cases except in the second factor in the second line where we write
$$ 1 + \frac{(N^{\ast})^{1/2}}{r^{1/2}} \leq 1 + \frac{(fgr^2/N)^{1/2 - \theta}(q^2/N)^{\theta}}{r^{1/2}} = 1 + \frac{(f g)^{1/2 - \theta} r^{1/2 - 2\theta}q^{2\theta}}{N^{1/2}}.$$
In the resulting expression we may substitute $f = g = 1$ by a divisor bound. Moreover the smallest power of $r$ is $r^{1 - 3/4}$ respectively $r^{1   - 1/2- 7/18}$ which in both cases is a positive $r$-power. Hence the expression is increasing in $r$ and we may substitute $r = q$, using again a divisor bound. 
In this way we obtain 
 \begin{equation}\label{bound3}
    \begin{split}
&\mathcal{B}^{\pm}(M, N) \ll q^{o(1)}    \min\Bigg(  \frac{M^{1/2}q^{1/4}}{N^{1/2}} + \frac{q^{1/2}}{N^{1/2}} + \frac{M^{1/2}}{q^{1/4}},  \Big(1 + \frac{M^{1/2}}{q^{1/2}}\Big)\Big(1 + \frac{q^{1/2}}{N^{1/2}}\Big) \\
&   \Big[\Big( \frac{M^2}{q^2} + \frac{M^3}{Nq} + \frac{M^2 q}{N^2} + \frac{M^3 q^2}{N^3}\Big)^{1/16}
\min\Big(\frac{q}{M}, q^{1/2}\Big)^{1/16} + \frac{q^{1/8}}{N^{1/8}} + \frac{M^{1/16}}{N^{1/32}q^{1/32}}  
    \\
    & 
 + \frac{q^{5/32} M^{1/16}}{N^{5/32}}  + 
    \frac{M^{1/3}}{q^{1/5}} + \frac{q^{7/15} }{N^{1/3}}
    +\frac{M^{1/2}}{q^{1/18}N^{1/6}} + \frac{M^{1/6}q^{11/18} }{N^{1/2}} + \frac{M^{1/15}}{q^{1/15}} + \frac{q^{1/15}}{N^{1/15}}
      \Big]  \Bigg).
\end{split}
\end{equation} 
  Combining the bounds \eqref{bound1}, \eqref{bound2}, \eqref{bound3} under the size constraint \eqref{ranges} leads to a linear optimization problem, which is performed most quickly  by a computer algebra system, such as {\tt mathematica}: 

 { \tt In[1] := Maximize[\{Min[7/64 n + 1/2 m + 1/2 n - 1, 
   Max[n/4 - m/4 - 1/4,\\  n/2 - m/2 - 1/2, -1/2 + 7/32], 
   Max[m/2 + 1/4 - n/2, 1/2 - n/2, m/2 - 1/4], \\Max[0, 1/2 - n/2] + Max[0, m/2 - 1/2] +  
   Max[1/16 Max[2m-2, 3m - n - 1, \\ 2m + 1 - 2n, 3m + 2 - 3n] + 1/16 Min[1-m, 1/2],  1/8 - n/8,   m/16 - n/32 - 1/32,\\ 5/32 - m/16 - 5n/32,   m/3 - 1/5,  7/15 - n/3, m/2 - 1/18 - n/6, 
    m/6 + 11/18 - n/2,\\ m/15 - 1/15,  1/15 - n/15]],  n >= m, m >= 0, 
  n >= 0, m + n <= 2\}, \{m, n\}] }

{\tt Out[1] := \{-1/90, \{m -> 43/90, n -> 43/30\}\}}

confirming the saving of $q^{-1/90+\varepsilon}$. This completes the proof of \cref{secondmoment}.


\section{The large sieve for exceptional Maa{\ss} forms} \label{sec:large-sieve}

In this section we prove \cref{thm:large-sieve}. We refer to \cite{deshouillers1982kloosterman,iwaniec1997topics,iwaniec2021spectral} for background on the spectral theory of automorphic forms. We are mainly interested in the Maa{\ss} cusp forms $f : \Gamma_0(q)\backslash \H \to \C$ of growing level $q \in \Z_+$ and bounded Laplacian eigenvalues. 
Following the original normalization of Deshouillers--Iwaniec \cite{deshouillers1982kloosterman}, we write the Fourier expansion of such a Maa{\ss} cusp form $f$ around a cusp $\ma$ of $\Gamma_0(q)\backslash \H$ as
\begin{equation}\label{eq:Maass-Fourier-expansion}
    f(\sigma_{\ma} z) = y^{1/2} \sum_{n \neq 0} \rho_{\ma}(n) K_{i\kappa}(2\pi |n| y)\, e(nx),
    \qquad\qquad 
    z = x + iy \in \H,
\end{equation}
where $\sigma_{\ma} \in \PSL_2(\R)$ is a scaling matrix for $\ma$ (satisfying in particular $\sigma_{\ma} \infty = \ma$), and $K$ is the Whittaker function as in \cite[p.\ 264]{deshouillers1982kloosterman}. We also recall \cref{not:orthonormal-basis} for an orthonormal basis of the discrete Maass spectrum.

The proof of \cref{thm:large-sieve} is very similar to, and in fact a bit simpler than, the proof of \cite[Theorem 9.4]{pascadi2025nonabelian}; the only new ingredient is our \cref{thm:bilinear-forms-balanced-lengths}. Note that for the choice of $X$ in \cref{eq:X-choice}, at least one of the following must be true:
\begin{itemize}
    \item[$(i)$] $X \le 8$,
    \item[$(ii)$] $X \ll \tfrac{q}{N} + \tfrac{q^2}{N^3}$, or
    \item[$(iii)$] $X \ll \min(\frac{q^{18/11}}{N^{23/11}}, \frac{q^{16/13}}{N^{18/13}}, \frac{q^{32/29}}{N^{33/29}})$.
\end{itemize}
If $X \le 8$, the result follows immediately from the regular-spectrum large sieve \cite[Theorem 2]{deshouillers1982kloosterman}. We further assume that $X > 8$, so one of $(ii)$ and $(iii)$ must hold. In particular, we have $\lfloor 2N \rfloor < \tfrac{NX}{4}$.

After adding the holomorphic and Eisenstein contributions and applying the Kuznetsov  formula (via \cite[Proposition 9.2]{pascadi2025nonabelian}), and separating the $m, n$ variables in the smooth weight, we obtain (identically as in \cite[p.\ 50]{pascadi2025nonabelian})
\begin{equation} \label{eq:ls-to-kloosterman}
\begin{aligned}
    \sum_{\lambda_j < 1/4}
    X^{2\theta_j} 
    \Big\vert 
    \sum_{n \sim N} \alpha_n\, \rho_{j\ma}(n)
    \Big\vert^2 
    &\ll 
    \frac{1}{NX}
    \sum_{\substack{NX/4 \le c \le NX \\ c \equiv 0 \pmod{q}}}
    \sup_{\substack{(\alpha_n')_{n \sim N},\ (\alpha_n'')_{n \sim N} \\ |\alpha'_n| = |\alpha''_n| = |\alpha_n|}}
    \Big\vert \sum_{m \sim N} \sum_{n \sim N} \alpha'_m \alpha''_n S(m, n; c) \Big\vert
    \\ 
    &+
    (qN)^{o(1)}
    \Big(1 + \frac{N}{q}\Big) \|\alpha\|^2,
\end{aligned}
\end{equation}
where the supremum is over all complex sequences $(\alpha'_n)_{n \sim N}$, $(\alpha''_n)_{n \sim N}$ with the same sequence of absolute values as $(\alpha_n)_{n \sim N}$. 

We first apply \cref{lem:bilinear-forms-trivial-bound} with $\mI = \mJ = \{1, \ldots, \lfloor 2N \rfloor\}$ to obtain
\[
    \sum_{m \sim N} \sum_{n \sim N} \alpha'_m \alpha''_n S(m, n; c) 
    \ll 
    (qN)^{o(1)} \|\alpha\|^2 \min(NX, N^{3/2} X^{1/2}).
\]
This gives an acceptable contribution to \cref{eq:ls-to-kloosterman} provided that
\[
    \frac{1}{q}\min(NX, N^{3/2} X^{1/2}) \ll 1
    \qquad 
    \iff 
    \qquad 
    X \ll \frac{q}{N} + \frac{q^2}{N^3},
\]
which covers case $(ii)$.

Finally, by \cref{thm:bilinear-forms-balanced-lengths} with $\mI = \mJ = \{1, \ldots, \lfloor 2N \rfloor\}$, we have
\[
    \sum_{m \sim N} \sum_{n \sim N} \alpha'_m \alpha''_n S(m, n; c) 
    \ll 
    (qN)^{o(1)}
    \|\alpha\|^2 NX \Big(\frac{N^{1/32}}{X^{3/32}} + \frac{N^{1/8}}{X^{3/16}} + \frac{N^{5/18}}{X^{7/18}}\Big).
\]
This gives an acceptable contribution to \cref{eq:ls-to-kloosterman} provided that
\[
    \frac{NX}{q} \Big(\frac{N^{1/32}}{X^{3/32}} + \frac{N^{1/8}}{X^{3/16}} + \frac{N^{5/18}}{X^{7/18}}\Big) \ll 1 
    \qquad 
    \iff 
    \qquad 
    X \ll 
    \min\Big(\frac{q^{32/29}}{N^{33/29}},
    \frac{q^{16/13}}{N^{18/13}},
    \frac{q^{18/11}}{N^{23/11}}
    \Big),
\]
which covers case $(iii)$. This completes our proof of \cref{thm:large-sieve}.\\

\textbf{Acknowledgement:} ChatGPT Pro was used to check for errors in an earlier version of this manuscript.

\bibliographystyle{plain}
\bibliography{main}

@article {moreno1991exponential,
    AUTHOR = {Moreno, Carlos J. and Moreno, Oscar},
     TITLE = {Exponential sums and {G}oppa codes. {I}},
   JOURNAL = {Proc. Amer. Math. Soc.},
  FJOURNAL = {Proceedings of the American Mathematical Society},
    VOLUME = {111},
      YEAR = {1991},
    NUMBER = {2},
     PAGES = {523--531},
      ISSN = {0002-9939,1088-6826},
   MRCLASS = {11T23 (11L40 14G15 94B40)},
  MRNUMBER = {1028291},
MRREVIEWER = {J.\ A.\ Thiong-Ly},
       DOI = {10.2307/2048345},
       URL = {https://doi.org/10.2307/2048345},
}

@article {maynard2025primes,
    AUTHOR = {Maynard, James},
     TITLE = {Primes in arithmetic progressions to large moduli {I}:
              fixed residue classes},
   JOURNAL = {Mem. Amer. Math. Soc.},
  FJOURNAL = {Memoirs of the American Mathematical Society},
    VOLUME = {306},
      YEAR = {2025},
    NUMBER = {1542},
     PAGES = {v+132},
      ISSN = {0065-9266,1947-6221},
      ISBN = {978-1-4704-7153-8; 978-1-4704-8050-9},
       DOI = {10.1090/memo/1542},
       URL = {https://doi.org/10.1090/memo/1542},
}

@article {maynard2025primes2,
    AUTHOR = {Maynard, James},
     TITLE = {Primes in arithmetic progressions to large moduli
              {II}: well-factorable estimates},
   JOURNAL = {Mem. Amer. Math. Soc.},
  FJOURNAL = {Memoirs of the American Mathematical Society},
    VOLUME = {306},
      YEAR = {2025},
    NUMBER = {1543},
     PAGES = {v+33},
      ISSN = {0065-9266,1947-6221},
      ISBN = {978-1-4704-7173-6; 978-1-4704-8051-6},
       DOI = {10.1090/memo/1543},
       URL = {https://doi.org/10.1090/memo/1543},
}

@article {maynard2025primes3,
    AUTHOR = {Maynard, James},
     TITLE = {Primes in arithmetic progressions to large moduli
              {III}: uniform residue classes},
   JOURNAL = {Mem. Amer. Math. Soc.},
  FJOURNAL = {Memoirs of the American Mathematical Society},
    VOLUME = {306},
      YEAR = {2025},
    NUMBER = {1544},
     PAGES = {v+98},
      ISSN = {0065-9266,1947-6221},
      ISBN = {978-1-4704-7154-5; 978-1-4704-8052-3},
       DOI = {10.1090/memo/1544},
       URL = {https://doi.org/10.1090/memo/1544},
}

@article {bombieri1986primes,
    AUTHOR = {Bombieri, Enrico and Friedlander, John B. and Iwaniec, Henryk},
     TITLE = {Primes in arithmetic progressions to large moduli},
   JOURNAL = {Acta Math.},
  FJOURNAL = {Acta Mathematica},
    VOLUME = {156},
      YEAR = {1986},
    NUMBER = {1},
     PAGES = {203--251},
      ISSN = {0001-5962,1871-2509},
       DOI = {10.1007/BF02399204},
       URL = {https://doi.org/10.1007/BF02399204},
}

@article {bombieri1987primes2,
    AUTHOR = {Bombieri, Enrico and Friedlander, John B. and Iwaniec, Henryk},
     TITLE = {Primes in arithmetic progressions to large moduli. {II}},
   JOURNAL = {Math. Ann.},
  FJOURNAL = {Mathematische Annalen},
    VOLUME = {277},
      YEAR = {1987},
    NUMBER = {3},
     PAGES = {361--393},
      ISSN = {0025-5831,1432-1807},
       DOI = {10.1007/BF01458321},
       URL = {https://doi.org/10.1007/BF01458321},
}

@article {bombieri1989primes3,
    AUTHOR = {Bombieri, Enrico and Friedlander, John B. and Iwaniec, Henryk},
     TITLE = {Primes in arithmetic progressions to large moduli. {III}},
   JOURNAL = {J. Amer. Math. Soc.},
  FJOURNAL = {Journal of the American Mathematical Society},
    VOLUME = {2},
      YEAR = {1989},
    NUMBER = {2},
     PAGES = {215--224},
      ISSN = {0894-0347,1088-6834},
       DOI = {10.2307/1990976},
       URL = {https://doi.org/10.2307/1990976},
}

@article {deshouillers1982kloosterman,
    AUTHOR = {Deshouillers, Jean-Marc and Iwaniec, Henryk},
     TITLE = {Kloosterman sums and {F}ourier coefficients of cusp forms},
   JOURNAL = {Invent. Math.},
  FJOURNAL = {Inventiones Mathematicae},
    VOLUME = {70},
      YEAR = {1982},
    NUMBER = {2},
     PAGES = {219--288},
      ISSN = {0020-9910,1432-1297},
       DOI = {10.1007/BF01390728},
       URL = {https://doi.org/10.1007/BF01390728},
}

@book {iwaniec2004analytic,
    AUTHOR = {Iwaniec, Henryk and Kowalski, Emmanuel},
     TITLE = {Analytic number theory},
    SERIES = {American Mathematical Society Colloquium Publications},
    VOLUME = {53},
 Publisher = {Amer. Math. Soc., Providence, RI},
      YEAR = {2004},
     PAGES = {xii+615},
      ISBN = {0-8218-3633-1},
   MRCLASS = {11-02 (11Fxx 11Lxx 11Mxx 11Nxx)},
  MRNUMBER = {2061214},
MRREVIEWER = {K.\ Soundararajan},
       DOI = {10.1090/coll/053},
       URL = {https://doi.org/10.1090/coll/053},
}

@incollection {selberg1965estimation,
    AUTHOR = {Selberg, Atle},
     TITLE = {On the estimation of {F}ourier coefficients of modular forms},
 BOOKTITLE = {Proc. {S}ympos. {P}ure {M}ath.},
    VOLUME = {8},
     PAGES = {1--15},
 Publisher = {Amer. Math. Soc., Providence, RI},
      YEAR = {1965},
}

@article {kim2003functoriality,
    AUTHOR = {Kim, Henry H.},
     TITLE = {Functoriality for the exterior square of {${\rm GL}_4$} and the symmetric fourth of {${\rm GL}_2$}},
      NOTE = {Appendix 2 by Henry H. Kim
              and Peter Sarnak},
   JOURNAL = {J. Amer. Math. Soc.},
  FJOURNAL = {Journal of the American Mathematical Society},
    VOLUME = {16},
      YEAR = {2003},
    NUMBER = {1},
     PAGES = {139--183},
      ISSN = {0894-0347,1088-6834},
       DOI = {10.1090/S0894-0347-02-00410-1},
       URL = {https://doi.org/10.1090/S0894-0347-02-00410-1},
}

@article {kuznetsov1980petersson,
    AUTHOR = {Kuznetsov, Nikolai V.},
     TITLE = {The {P}etersson conjecture for cusp forms of weight zero and the {L}innik conjecture. {S}ums of {K}loosterman sums},
   JOURNAL = {Mat. Sb. (N.S.)},
  FJOURNAL = {Matematicheski\u i\ Sbornik. Novaya Seriya},
    VOLUME = {111(153)},
      YEAR = {1980},
    NUMBER = {3},
     PAGES = {334--383},
      ISSN = {0368-8666},
}

@article {de2020niveau,
    AUTHOR = {de La Bret{\`e}che, R{\'e}gis and Drappeau, Sary},
     TITLE = {Niveau de r\'epartition des polyn\^omes quadratiques et crible majorant pour les entiers friables},
   JOURNAL = {J. Eur. Math. Soc. (JEMS)},
  FJOURNAL = {Journal of the European Mathematical Society (JEMS)},
    VOLUME = {22},
      YEAR = {2020},
    NUMBER = {5},
     PAGES = {1577--1624},
      ISSN = {1435-9855,1435-9863},
       DOI = {10.4171/jems/951},
       URL = {https://doi.org/10.4171/jems/951},
}

@article {kowalski2017bilinear,
    AUTHOR = {Kowalski, Emmanuel and Michel, Philippe and Sawin, Will},
     TITLE = {Bilinear forms with {K}loosterman sums and applications},
   JOURNAL = {Ann. of Math. (2)},
  FJOURNAL = {Annals of Mathematics. Second Series},
    VOLUME = {186},
      YEAR = {2017},
    NUMBER = {2},
     PAGES = {413--500},
      ISSN = {0003-486X,1939-8980},
       DOI = {10.4007/annals.2017.186.2.2},
       URL = {https://doi.org/10.4007/annals.2017.186.2.2},
}

@article {fouvry2014algebraic,
    AUTHOR = {Fouvry, \'Etienne and Kowalski, Emmanuel and Michel, Philippe},
     TITLE = {Algebraic trace functions over the primes},
   JOURNAL = {Duke Math. J.},
  FJOURNAL = {Duke Mathematical Journal},
    VOLUME = {163},
      YEAR = {2014},
    NUMBER = {9},
     PAGES = {1683--1736},
      ISSN = {0012-7094,1547-7398},
       DOI = {10.1215/00127094-2690587},
       URL = {https://doi.org/10.1215/00127094-2690587},
}

@article {merikoski2023largest,
    AUTHOR = {Merikoski, Jori},
     TITLE = {On the largest prime factor of {$n^2+1$}},
   JOURNAL = {J. Eur. Math. Soc. (JEMS)},
  FJOURNAL = {Journal of the European Mathematical Society (JEMS)},
    VOLUME = {25},
      YEAR = {2023},
    NUMBER = {4},
     PAGES = {1253--1284},
      ISSN = {1435-9855,1435-9863},
       DOI = {10.4171/jems/1216},
       URL = {https://doi.org/10.4171/jems/1216},
}

@article {kowalski2020stratification,
    AUTHOR = {Kowalski, Emmanuel and Michel, Philippe and Sawin, Will},
     TITLE = {Stratification and averaging for exponential sums: bilinear
              forms with generalized {K}loosterman sums},
   JOURNAL = {Ann. Sc. Norm. Super. Pisa Cl. Sci. (5)},
  FJOURNAL = {Annali della Scuola Normale Superiore di Pisa. Classe di
              Scienze. Serie V},
    VOLUME = {21},
      YEAR = {2020},
     PAGES = {1453--1530},
      ISSN = {0391-173X,2036-2145},
       DOI = {10.2422/2036-2145.201805\_002},
       URL = {https://doi.org/10.2422/2036-2145.201805_002},
}

@book {iwaniec2021spectral,
    AUTHOR = {Iwaniec, Henryk},
     TITLE = {Spectral methods of automorphic forms},
    SERIES = {Graduate Studies in Mathematics},
    VOLUME = {53},
   EDITION = {Second},
 Publisher = {Amer. Math. Soc., Providence, RI; Revista
              Matem\'atica Iberoamericana, Madrid},
      YEAR = {2002},
     PAGES = {xii+220},
      ISBN = {0-8218-3160-7},
       DOI = {10.1090/gsm/053},
       URL = {https://doi.org/10.1090/gsm/053},
}

@book {iwaniec1997topics,
    AUTHOR = {Iwaniec, Henryk},
     TITLE = {Topics in classical automorphic forms},
    SERIES = {Graduate Studies in Mathematics},
    VOLUME = {17},
 Publisher = {Amer. Math. Soc., Providence, RI},
      YEAR = {1997},
     PAGES = {xii+259},
      ISBN = {0-8218-0777-3},
       DOI = {10.1090/gsm/017},
       URL = {https://doi.org/10.1090/gsm/017},
}

@article {deshouillers1982greatest,
    AUTHOR = {Deshouillers, Jean-Marc and Iwaniec, Henryk},
     TITLE = {On the greatest prime factor of {$n\sp{2}+1$}},
   JOURNAL = {Ann. Inst. Fourier (Grenoble)},
  FJOURNAL = {Universit\'e{} de Grenoble. Annales de l'Institut Fourier},
    VOLUME = {32},
      YEAR = {1982},
    NUMBER = {4},
     PAGES = {1--11},
      ISSN = {0373-0956,1777-5310},
       DOI = {10.5802/aif.891},
       URL = {https://doi.org/10.5802/aif.891},
}

@article {topacogullari2018shifted,
    AUTHOR = {Topacogullari, Berke},
     TITLE = {The shifted convolution of generalized divisor functions},
   JOURNAL = {Int. Math. Res. Not. IMRN},
  FJOURNAL = {International Mathematics Research Notices. IMRN},
      YEAR = {2018},
    VOLUME = {2018},
    NUMBER = {24},
     PAGES = {7681--7724},
      ISSN = {1073-7928,1687-0247},
       DOI = {10.1093/imrn/rnx111},
       URL = {https://doi.org/10.1093/imrn/rnx111},
}

@article {deshouillers1984power,
    AUTHOR = {Deshouillers, Jean-Marc and Iwaniec, Henryk},
     TITLE = {Power mean-values for {D}irichlet's polynomials and the
              {R}iemann zeta-function. {II}},
   JOURNAL = {Acta Arith.},
  FJOURNAL = {Polska Akademia Nauk. Instytut Matematyczny. Acta Arithmetica},
    VOLUME = {43},
      YEAR = {1984},
    NUMBER = {3},
     PAGES = {305--312},
      ISSN = {0065-1036},
       DOI = {10.4064/aa-43-3-305-312},
       URL = {https://doi.org/10.4064/aa-43-3-305-312},
}

@article {deshouillers1982power,
    AUTHOR = {Deshouillers, Jean-Marc and Iwaniec, Henryk},
     TITLE = {Power mean values of the {R}iemann zeta function},
   JOURNAL = {Mathematika},
  FJOURNAL = {Mathematika. A Journal of Pure and Applied Mathematics},
    VOLUME = {29},
      YEAR = {1982},
    NUMBER = {2},
     PAGES = {202--212},
      ISSN = {0025-5793},
   MRCLASS = {10H05},
  MRNUMBER = {696876},
MRREVIEWER = {K.\ Ramachandra},
       DOI = {10.1112/S0025579300012298},
       URL = {https://doi.org/10.1112/S0025579300012298},
}

@article {lichtman2025modification,
    AUTHOR = {Lichtman, Jared D.},
     TITLE = {A modification of the linear sieve, and the count of twin
              primes},
   JOURNAL = {Algebra Number Theory},
  FJOURNAL = {Algebra \& Number Theory},
    VOLUME = {19},
      YEAR = {2025},
    NUMBER = {1},
     PAGES = {1--38},
      ISSN = {1937-0652,1944-7833},
       DOI = {10.2140/ant.2025.19.1},
       URL = {https://doi.org/10.2140/ant.2025.19.1},
}

@article {pascadi2026large,
    AUTHOR = {Pascadi, Alexandru},
     TITLE = {Large sieve inequalities for exceptional {M}aass forms and the
              greatest prime factor of {$n^2+1$}},
   JOURNAL = {Forum Math. Pi},
  FJOURNAL = {Forum of Mathematics. Pi},
    VOLUME = {14},
      YEAR = {2026},
     PAGES = {e8},
      ISSN = {2050-5086},
   MRCLASS = {11N75 (11L05 11N32)},
  MRNUMBER = {5035407},
       DOI = {10.1017/fmp.2026.10025},
       URL = {https://doi.org/10.1017/fmp.2026.10025},
}

@article {blomer2017moments,
    AUTHOR = {Blomer, Valentin and Fouvry, \'Etienne and Kowalski, Emmanuel and Michel, Philippe and Mili\'cevi\'c, Djordje},
     TITLE = {On moments of twisted {$L$}-functions},
   JOURNAL = {Amer. J. Math.},
  FJOURNAL = {American Journal of Mathematics},
    VOLUME = {139},
      YEAR = {2017},
    NUMBER = {3},
     PAGES = {707--768},
      ISSN = {0002-9327,1080-6377},
   MRCLASS = {11M06 (11F11 11F72 11L05 11L40 11T23)},
  MRNUMBER = {3650231},
MRREVIEWER = {Arnaud\ Chadozeau},
       DOI = {10.1353/ajm.2017.0019},
       URL = {https://doi.org/10.1353/ajm.2017.0019},
}

@article {blomer2015second,
    AUTHOR = {Blomer, Valentin and Mili\'cevi\'c, Djordje},
     TITLE = {The second moment of twisted modular {$L$}-functions},
   JOURNAL = {Geom. Funct. Anal.},
  FJOURNAL = {Geometric and Functional Analysis},
    VOLUME = {25},
      YEAR = {2015},
    NUMBER = {2},
     PAGES = {453--516},
      ISSN = {1016-443X,1420-8970},
   MRCLASS = {11F66 (11F72 11L07)},
  MRNUMBER = {3334233},
MRREVIEWER = {Ravi\ Raghunathan},
       DOI = {10.1007/s00039-015-0318-7},
       URL = {https://doi.org/10.1007/s00039-015-0318-7},
}

@article{grimmelt2025greatest,
  title={{On the greatest prime factor and uniform equidistribution of quadratic polynomials}},
  author={Grimmelt, Lasse and Merikoski, Jori},
  journal={Preprint, arXiv:2505.00493},
  year={2025}
}

@article{milicevic2025bilinear,
  title={Bilinear forms with {K}loosterman sums and moments of twisted {$L$}-functions},
  author={Mili\'cevi\'c, Djordje and Qin, Xinhua and Wu, Xiaosheng},
  journal={Preprint, arXiv:2511.07550},
  year={2025}
}

@book {schmidt1976bilinear,
    AUTHOR = {Schmidt, Wolfgang M.},
     TITLE = {Equations over finite fields. {A}n elementary approach},
    SERIES = {Lecture Notes in Mathematics},
    VOLUME = {Vol. 536},
 PUBLISHER = {Springer-Verlag, Berlin-New York},
      YEAR = {1976},
     PAGES = {ix+276},
   MRCLASS = {10B15 (10G05 14G13)},
  MRNUMBER = {429733},
MRREVIEWER = {H.\ M.\ Stark},
}

@article {pascadi2025exponents,
  title={On the exponents of distribution of primes and smooth numbers},
  author={Pascadi, Alexandru},
  journal={Preprint, arXiv:2505.00653},
  year={2025}
}

@article {pascadi2025nonabelian,
  title={Non-abelian amplification and bilinear forms with Kloosterman sums},
  author={Pascadi, Alexandru},
  journal={Geom. Funct. Anal., accepted. Preprint, arXiv:2511.08445v2},
  year={2025}
}

@article {fouvry2025bilinear,
  title={Bilinear forms with trace functions},
  author={Fouvry, {\'E}tienne and Kowalski, Emmanuel and Michel, Philippe and Sawin, Will},
  journal={Preprint, arXiv:2511.09459v1},
  year={2025}
}

@article {chandee2024eighth,
    AUTHOR = {Chandee, Vorrapan and Li, Xiannan and Matom\"aki, Kaisa and
              Radziwi\l\l, Maksym},
     TITLE = {The eighth moment of {D}irichlet {$L$}-functions {II}},
   JOURNAL = {Duke Math. J.},
  FJOURNAL = {Duke Mathematical Journal},
    VOLUME = {173},
      YEAR = {2024},
    NUMBER = {18},
     PAGES = {3453--3493},
      ISSN = {0012-7094,1547-7398},
   MRCLASS = {11M06 (11M26)},
  MRNUMBER = {4850958},
MRREVIEWER = {Andrius\ Grigutis},
       DOI = {10.1215/00127094-2024-0014},
       URL = {https://doi.org/10.1215/00127094-2024-0014},
}

\end{document}